\documentclass[11pt,a4paper]{article}
\usepackage[margin=25mm]{geometry}
\usepackage{amsmath,amssymb}
\usepackage{amsthm}
\usepackage{mathrsfs}
\usepackage{graphicx}
\usepackage{wrapfig}
\numberwithin{equation}{section}
\newtheorem{theorem}{Theorem}[section]
\newtheorem{lemma}[theorem]{Lemma}
\newtheorem{proposition}[theorem]{Proposition}
\newtheorem{step}{Step}[section]
\newtheorem{corollary}[theorem]{Corollary}
\newtheorem{assumption}[theorem]{Assumption}
\newtheorem{example}[theorem]{Example}
\newtheorem{app}{Appendix}[section]
\newtheorem{remark}{Remark}[section]
\begin{document}

\title{Global Well-Posedness for the Fourth-Order Nonlinear Schr\"{o}dinger Equation with Potential in the Energy-Critical Case}
\author{Hikaru Nakayama\thanks{This paper is based on the author's master's thesis completed at the Department of Mathematics, Graduate School of Science, Osaka University.}}
\date{}
\maketitle

\begin{abstract}
We consider the fourth-order nonlinear Schr\"{o}dinger equation with potential in the energy-critical case.
In this paper, we prove global well-posedness and scattering for radial initial data.
\end{abstract}

\tableofcontents

\newpage

\section{Introduction and Main Result}

\subsection{Introduction}

For $\lambda \in \mathbb{R}$, we consider the following fourth-order nonlinear Schr\"{o}dinger equation:
\begin{equation}\label{eq:NL4S}
  i\partial_{t}u + \Delta^{2}u + Vu + \lambda|u|^{p-1}u = 0,\ \ x \in \mathbb{R}^{n}, \ \ t \in \mathbb{R}.
\end{equation}
Here $\Delta = \sum_{j = 1}^{n} \partial_{j}^{2}$ and $\Delta^{2} = \Delta \circ \Delta$. Moreover, $V$ is a real-valued potential. Define $2^{\sharp} = 2n/(n-4)$. When $p = 2^{\sharp} - 1$, the equation is called energy-critical, and when $1 < p < 2^{\sharp}-1$, it is called energy-subcritical. Also, when $\lambda > 0$, the equation is called defocusing. Throughout this paper, $\langle \cdot, \cdot \rangle$ denotes the $L^{2}$ inner product, and $\langle  \cdot \rangle = (1 + |\cdot|^{2})^{1/2}$. We say that $u \in C(I, H^{2})$ is a solution to \eqref{eq:NL4S} with initial data $u_{0} = u(0)$ if it satisfies the integral equation
\[
 u(t) = e^{itH}u_{0} - i\int_{0}^{t}e^{i(t-s)H}|u|^{p-1}u(s)\, ds,
\]
where $H = \Delta^{2} + V$.

Pausader \cite{Pau} proved global well-posedness and scattering for radial initial data in the energy-critical case for the fourth-order nonlinear Schr\"{o}dinger equation obtained from \eqref{eq:NL4S} by taking $V = 0$. In this paper, using the method of \cite{Pau}, we obtain an analogous result for the fourth-order nonlinear Schr\"{o}dinger equation with potential \eqref{eq:NL4S}. In \cite{Pau}, Strichartz estimates play an essential role in the proof. Therefore, in the present paper we need Strichartz estimates with potential, which were established in \cite{Miz}. However, when a potential is present, the propagator $e^{it(\Delta^{2} + V)}$ cannot in general be represented by the Fourier transform. Thus, in order to obtain Strichartz estimates in Sobolev norms, we need the equivalence of Sobolev norms
\[
\|(\Delta^{2} + V)^{s/4}u\|_{L^{p}} \sim \||\nabla|^{s}u\|_{L^{p}}.
\]
This is a major difference from the case $V = 0$. We discuss this equivalence in detail in Section 3, and by using it we derive Strichartz estimates in Sobolev norms in Section 4. In addition, \cite{Pau} uses $L^{p}$-$L^{q}$ decay estimates; in order to obtain their counterparts in the presence of a potential, we use the boundedness of wave operators, which is also treated in Section 4. Furthermore, in Section 8 we derive a Morawetz-type estimate with potential, which is needed to prove the existence of global solutions. Then, in Section 9, we prove global existence. Moreover, for nonlinear Schr\"{o}dinger equations with potential, besides the question of whether a solution $u$ asymptotically approaches a linear solution $e^{it(\Delta^{2} + V)}u^{\pm}$, it is also important to ask whether it asymptotically approaches a free linear solution $e^{it\Delta^{2}}u^{\pm}$ corresponding to the case $V = 0$. This is discussed in Section 10 on scattering.

\subsection{Main result}

In order to state the main theorems of this paper, we impose the following assumptions.

\begin{assumption}\label{as:wave}
Let $V$ be a real-valued potential, and let $\mathcal{F}$ denote the Fourier transform. Assume that $n \geq 5$.

\begin{enumerate}
  \item $|V(x)| \lesssim \langle x \rangle^{-\beta}$ holds for some $\beta > n + 3$ when $n$ is odd, and for some $\beta > n + 4$ when $n$ is even.
  \item When $n = 7$, there exists $\varepsilon > 0$ such that $\|\langle \cdot \rangle^{1 + \varepsilon} V(\cdot)\|_{H^{\varepsilon}} < \infty$.
  \item There exist $0 < \delta \ll 1$ and $\sigma > \frac{2n - 8}{n - 1 - \delta}$ such that
  \[
  \|\mathcal{F}(\langle \cdot \rangle^{\sigma}V(\cdot))\|_{L^{\frac{n - 1 - \delta}{n - 4 - \delta}}} < \infty.
  \]
  \item $H = \Delta^{2} + V$ has no positive eigenvalues, and zero energy is regular.
\end{enumerate}
Here, ``zero energy is regular'' means that there is no nontrivial solution to $H\psi = 0$ satisfying $\langle x \rangle^{\frac{n}{2} - 4 - \varepsilon}\psi(x) \in L^{2}$.
\end{assumption}

\begin{assumption}\label{as:mora}
Let $V$ be a $C^{1}$ real-valued potential.
\begin{enumerate}
\item $x \cdot \nabla V \leq 0$.
\item $|\partial^{\alpha}_{x}V(x)| \leq C_{\alpha}\langle x \rangle^{-\beta-|\alpha|}$ for $|\alpha| = 0 , 1$ and some $\beta \geq 4$.
\end{enumerate}
\end{assumption}

\begin{assumption}\label{as:equ}
Let $V$ be a real-valued potential.
\begin{enumerate}
\item $V \geq 0$.
\item $V \in L^{\infty} \cap L^{\frac{n}{4}}$.
\item There exists a small $\delta_{n} > 0$ such that $\|V\|_{L^{\frac{n}{4},\infty}} \leq \delta_{n}$.
\end{enumerate}
\end{assumption}

The following two theorems are the main results of this paper.

\begin{theorem}\label{th:global}
Let $n \geq 5$, $\lambda > 0$, and $p \in (1, \, 2^{\sharp} - 1]$. Let $V \in L^{\infty} \cap L^{\frac{n}{4}}$ be a radial real-valued potential satisfying Assumptions \ref{as:wave}, \ref{as:mora}, and \ref{as:equ}. Then, for any radial initial data $u_{0} \in H^{2}$, there exists a unique global solution $u\in C(\mathbb{R}, H^{2})$ to \eqref{eq:NL4S} satisfying $u(0) = u_{0}$.
\end{theorem}

\begin{theorem}\label{th:scattering}
Let $n \geq 5$, and let $V$ be a radial real-valued potential satisfying Assumptions \ref{as:wave}, \ref{as:mora}, and \ref{as:equ}. Let $u \in C(\mathbb{R}, H^{2})$ be any radial solution to \eqref{eq:NL4S} with $p = 2^{\sharp}-1$ and $\lambda > 0$. Then there exist radial functions $u^{\pm} \in H^{2}$ such that
\[
 \|u(t) - e^{it\Delta^{2}}u^{\pm} \|_{H^{2}}\rightarrow 0 ,\ \ t \rightarrow \pm \infty.
\]
\end{theorem}

These results extend the result of \cite{Pau}, where \eqref{eq:NL4S} was studied in the case $V = 0$. Assumption \ref{as:wave} is used to derive the $L^{p}$-$L^{q}$ decay estimates introduced in Section 4, Assumption \ref{as:mora} is used to derive the Morawetz-type estimate introduced in Section 8, and Assumption \ref{as:equ} is used to establish the equivalence of Sobolev norms
\[
\|H^{s/4}u\|_{L^{p}} \sim \||\nabla|^{s}u\|_{L^{p}},
\]
which is discussed in Section 3.

The following gives an example of a potential satisfying the assumptions of the main theorems.

\begin{example}
Let $\beta > n + 4$ and let $0 < c \ll 1$ be sufficiently small. Define
\[
 V(x) = c\langle x \rangle^{-\beta}.
\]
Then it is easy to see that $V$ is radial and satisfies Assumptions \ref{as:mora} and \ref{as:equ}. It is also immediate that Assumption \ref{as:wave}(1) and (2) hold. For Assumption \ref{as:wave}(3), by the Hausdorff--Young inequality, using a sufficiently small constant $0 < \varepsilon \ll 1$, we obtain
\[
 \|\mathcal{F}\left( \langle x \rangle^{\frac{2n-8}{n-1} -\beta + \varepsilon} \right)\|_{L^{1 + \frac{3}{n-4} + \varepsilon}} \lesssim \|\langle x \rangle^{\frac{2n-8}{n-1} -\beta + \varepsilon}\|_{L^{(1 + \frac{3}{n-4})^{\prime} - \varepsilon}} < \infty,
\]
so Assumption \ref{as:wave}(3) is satisfied. Assumption \ref{as:wave}(4) follows from Remark \ref{remark:st:wave:as}.
\end{example}

\subsection{Related background}
For the second-order nonlinear Schr\"{o}dinger equation
\[
  i\partial_{t}u + \Delta u + \lambda|u|^{p-1}u = 0,\ \ x \in \mathbb{R}^{n}, \ \ t \in \mathbb{R},
\]
there are many results in the energy-critical case $p = \frac{n + 2}{n-2}$ and the defocusing case, namely $\lambda < 0$. Historically, Bourgain \cite{Bourgain} first proved global well-posedness for radial solutions when $n = 3,4$, and then Tao \cite{Tao-2} generalized Bourgain's result to arbitrary $n \geq 3$. Furthermore, J. Colliander, M. Keel, G. Staffilani, H. Takaoka, and T. Tao \cite{Coll} obtained global well-posedness for general initial data when $n = 3$, and finally Ryckman and Visan \cite{Ryckman}, \cite{visan} extended this to $n \geq 4$. Thus, the study of global well-posedness for the usual second-order nonlinear Schr\"{o}dinger equation has the above background.

The fourth-order nonlinear Schr\"{o}dinger equation was introduced by Karpman \cite{Karpman} and by Karpman and Shagalov \cite{Karpman-2} in order to describe fourth-order dispersion effects for intense laser beams propagating through a medium with Kerr nonlinearity. Subsequently, well-posedness for fourth-order nonlinear Schr\"{o}dinger equations was studied by Pausader \cite{Pau}, \cite{Pau-2}, and by C. Hao, L. Hsiao, and B. Wang \cite{Hao-1}, \cite{Hao-2}, among others. In particular, global well-posedness and scattering in the energy-critical case were investigated by Pausader \cite{Pau}, \cite{Pau-3}, and by Miao, Xu, and Zhao \cite{Miao}, \cite{Miao-2}. More recently, Miao and Zheng \cite{Miao-3}, as well as Dinh \cite{Dinh-1}, \cite{Dinh-2}, studied the $\dot{H}^{s_{c}}$-critical problem and obtained results on scattering theory. Here,
\[
 s_{c} = \frac{n}{2} - \frac{4}{p - 1}.
\]

As for previous work on the fourth-order nonlinear Schr\"{o}dinger equation with potential, namely the case $V \neq 0$ in \eqref{eq:NL4S}, Feng, Wang, and Yao \cite{Feng} proved global well-posedness and scattering in the energy-subcritical case.

\section{Notations}
Let $L^{q} = L^{q}(\mathbb{R}^{n})$ denote the usual Lebesgue space, and let $L^{r}(I, L^{q})$ denote the space of measurable functions from an interval $I \subset \mathbb{R}$ to $L^{q}$ such that the norm
\[
 \|u\|_{L^{r}(I, L^{q})} = \left( \int_{I} \|u(t)\|_{L^{q}}^{r}dt \right)^{1/r}
\]
is finite.

Next we define generalized Sobolev spaces. For $s\in \mathbb{R}$ and $1 \leq q \leq \infty$, define
\[
 H^{s, q} := \left\{u \in \mathcal{S}^{\prime} :  \|u\|_{H^{s, q}} := \|\langle \Lambda \rangle^{s} u\|_{L^{q}} < \infty \right\},\ \ \Lambda = \sqrt{-\Delta}.
\]
Here $\mathcal{S}^{\prime}$ denotes the space of tempered distributions. In particular, we write $H^{s, 2} = H^{s}$. We also denote by $L^{r, \infty}\ (1 \leq r < \infty)$ the weak-$L^{r}$ space, defined by
\[
 L^{r, \infty} = \left\{ f: \|f\|_{L^{r, \infty}} = \sup_{\gamma > 0}\gamma \left|\{x \in \mathbb{R}^{n} ; |f(x)| > \gamma \}^{1/\gamma} \right| < \infty \right\}.
\]
For later use, we recall the weak-type H\"{o}lder inequality. If $\dfrac{1}{p} = \dfrac{1}{q} + \dfrac{1}{r}$ with $1 < p, q , r < \infty$, and if $f\in L^{r, \infty}$ and $g \in L^{q}$, then
\begin{equation}\label{weak holder}
\|fg\|_{L^{p}} \lesssim_{p,q,r} \|f\|_{L^{r, \infty}}\|g\|_{L^{q}}.
\end{equation}

Here $ A\lesssim_{u}B$ means that there exists a constant $C > 0$, depending only on $u$, such that $A \leq CB$. Next we define two important conserved quantities, the mass $M(u)$ and the energy $E(u)$:
\begin{equation}\label{eq:mass}
M(u) = \int_{\mathbb{R}^{n}} |u(x)|^{2} dx
\end{equation}

\begin{equation}\label{eq:energy}
E(u) = \frac{1}{2}\int_{\mathbb{R}^{n}} \left(|\Delta u(x)|^{2} + V(x)|u(x)|^{2} + \frac{2\lambda}{p + 1}  |u(x)|^{p + 1}\right)\, dx
\end{equation}

We introduce the norms used in the following arguments:
\begin{eqnarray*}
\|u\|_{M(I)} &=& \|\Delta u\|_{L^{\frac{2(n + 4)}{n - 4} }(I, L^{\frac{2n(n+4)}{n^{2} + 16}})}\\
\|u\|_{W(I)} &=& \|\nabla u\|_{L^{\frac{2(n + 4)}{n - 4} }(I, L^{\frac{2n(n+4)}{n^{2} - 2n + 8}})} \\
\|u\|_{Z(I)} &=& \|u\|_{L^{\frac{2(n + 4)}{n - 4} }(I, L^{\frac{2(n + 4)}{n - 4}})} \\
\|u\|_{N(I)} &=& \|\nabla u\|_{L^{2}(I, L^{\frac{2n}{n + 2}})}.
\end{eqnarray*}

For $u\in H^{2}$, define the functional $\mathcal{E}$ by
\begin{equation}
\mathcal{E}(u) = \int_{\mathbb{R}^{n}} |\Delta u(x)|^{2} dx.
\end{equation}

Let $\mathbb{B}(X, Y)$ denote the space of bounded operators from $X$ to $Y$, and in particular write $\mathbb{B}(X, X) = \mathbb{B}(X)$. We denote the operator norm by $\|\cdot\|_{X \rightarrow Y} := \|\cdot\|_{\mathbb{B}(X, Y)}$. Also, if $H$ is a self-adjoint operator, then $\mathcal{H}_{ac}(H)$ denotes the absolutely continuous spectral subspace of $H$, and $P_{ac}(H)$ denotes the projection onto $\mathcal{H}_{ac}(H)$. We also write $\sigma(H)$ for the spectrum of $H$.

\section{Equivalence of Homogeneous Sobolev Norms on $L^{p}$}

\begin{assumption}\label{as:sobolev}
Let $q_{0} > 2$. Assume that there exist constants $a, b > 0$ such that
\[\|Vf\|_{L^{q_{0}}} \leq a\|\Delta^{2}f\|_{L^{q_{0}}} + b\|f\|_{L^{q_{0}}},\ \ f \in H^{4, q_{0}}. \]
\end{assumption}

By \cite{Deng-2}, the following $L^{p}$ boundedness of the Riesz transform is known.
\begin{proposition}\label{riesz}
Let $n \geq 5$ and set $H = \Delta^{2} + V$. Assume that $V$ satisfies Assumption \ref{as:sobolev} for some $2 < q_{0} < \frac{n}{2}$. Then there exists $\delta_{q_{0}} > 0$ such that if
\begin{equation}
\|V\|_{L^{\frac{n}{4}, \infty}} \leq \delta_{q_{0}},
\end{equation}
then the Riesz transform $\nabla^{2}H^{-1/2}$ is bounded on $L^{q}$ for every $\dfrac{nq_{0}^{\prime}}{n + 2q_{0}^{\prime}} < q < q_{0}$.
\end{proposition}

\begin{assumption}\label{as:kernel}
Let $V$ be a real-valued potential. Assume that for every $\varepsilon > 0$ there exists a constant $C_{\varepsilon} > 0$ such that
\[ \|Vf\|_{L^{1}} \leq \varepsilon \|\Delta^{2}f\|_{L^{1}} + C_{\varepsilon}\|f\|_{L^{1}}, \ \ f \in C_{0}^{\infty}(\mathbb{R}^{n}).\]
\end{assumption}

\begin{proposition}\label{prop:h equ}
Let $n \geq 5$ and $H = \Delta^{2} + V$. Assume that $V \in L^{\frac{n}{4}}$ and $V \geq 0$. Then the following two statements hold.
\begin{itemize}
\item If $V \in L^{\infty}$ and the weak $L^{\frac{n}{4}}$ norm of $V$ is sufficiently small, then for every $1 < p < \dfrac{n}{2}$ and $0 \leq s \leq 2$,
\begin{equation}\label{eq1}
  \||\nabla |^{s}f\|_{L^{p}} \lesssim \|H^{s/4}f\|_{L^{p}}.
\end{equation}
\item Assume that $V$ satisfies Assumption \ref{as:kernel}. Then for every $1 < p < \dfrac{n}{s}$ and $0 \leq s \leq 4$,
\begin{equation}\label{eq2}
  \|H^{s/4}f\|_{L^{p}} \lesssim \||\nabla |^{s}f\|_{L^{p}}.
\end{equation}
\end{itemize}
\end{proposition}

\begin{proof}
First we prove \eqref{eq1}. Since $V\in L^{\infty}$, Assumption \ref{as:sobolev} holds for any $2 < q_{0} < \frac{n}{2}$. Therefore, by applying Proposition \ref{riesz}, we obtain for $1 < p < \frac{n}{2}$ that
\begin{equation}\label{riesz:s=2}
  \||\nabla |^{2}f\|_{L^{p}} \lesssim \|H^{1/2}f\|_{L^{p}}.
\end{equation}
Hence it remains to prove the case $0 < s < 2$. For $x,y \in \mathbb{R}$, set
\[ T_{x + iy} := |\nabla|^{2x + iy}H^{-\frac{x}{2} + iy}. \]
Then we show the $L^{p}$ boundedness of $T_{1 + iy} = |\nabla|^{iy}|\nabla|^{2}H^{-\frac{1}{2}}H^{iy}$. By \eqref{riesz:s=2}, it suffices to prove the $L^{p}$ boundedness of $H^{iy}$. Since $V \in L^{\infty}$ by assumption, Assumption \ref{as:kernel} is satisfied. Hence, by \cite{Deng-1}, the semigroup $e^{-tH}$ extends to an analytic semigroup $e^{zH}$ on $L^{1}$ with angle $\frac{\pi}{2}$, and when $V \geq 0$ its kernel satisfies
\begin{equation}
|e^{-tH}(x, y)| \lesssim t^{-\frac{n}{4}} \exp \left \{ -c\dfrac{|x - y|^{\frac{4}{3}}}{t^{\frac{1}{3}}} \right\}.
\end{equation}
Therefore, by applying Blunck's abstract spectral multiplier theorem \cite{Blunk} to $H$, we obtain for $1 < p < \infty$ that
\[ \|H^{iy}\|_{L^{p}\rightarrow L^{p}} \leq C_{p}\langle y \rangle^{2}. \]
Hence the case $0 < s < 2$ follows from Stein's interpolation theorem \cite{Stein}.

Next we prove \eqref{eq2}. Set $H_{0} = \Delta^{2}$ and define
\[ T_{z} = H^{z}H^{-z}_{0},\ \ z = x + iy. \]
We already know that $H^{iy}$ and $H^{-iy}_{0}$ are bounded on $L^{q}$ for $1 < q < \infty$. Thus, taking $z = iy$, we have for $1 < q < \infty$,
\begin{equation}\label{iy}
\|T_{iy}\|_{L^{q} \rightarrow L^{q}} \lesssim \langle y \rangle^{\alpha}.
\end{equation}
Next take $z = 1 + iy$. For $1 < p < \frac{n}{4}$, by H\"{o}lder and Sobolev inequalities,
\begin{align*}
\|Hf\|_{L^{p}} &\leq \|H_{0}f\|_{L^{p}} + \|Vf\|_{L^{p}} \\
&\lesssim \|H_{0}f\|_{L^{p}} + \|V\|_{L^{\frac{n}{4}}} \|H_{0}f\|_{L^{p}} \\
&\lesssim \|H_{0}f\|_{L^{p}}.
\end{align*}
Therefore, for $1 < p < \frac{n}{4}$,
\begin{equation}\label{1+iy}
\|T_{1 + iy}\|_{L^{p} \rightarrow L^{p}} \lesssim \langle y \rangle^{\alpha}
\end{equation}
holds. Hence, by \eqref{iy}, \eqref{1+iy}, and Stein's interpolation theorem, for $\theta \in [0, 1]$,
\[ \|T_{\theta}\|_{L^{r} \rightarrow L^{r}} \lesssim 1 \]
holds, where $1 < r < \frac{n}{4\theta}$. Writing $s = 4\theta$, we obtain for $s \in [0, 4]$ and $1 < r < \frac{n}{s}$ that
\[ \|H^{s/4}f\|_{L^{r}} \lesssim \||\nabla |^{s}f\|_{L^{r}}. \]
\end{proof}

By Proposition \ref{prop:h equ}, under Assumption \ref{as:equ}, since $V \in L^{\infty}$ implies Assumption \ref{as:kernel}, we obtain the equivalence of Sobolev norms
\[\|H^{s/4}u\|_{L^{p}} \sim \||\nabla|^{s}u\|_{L^{p}} \]
for any $1 < p < \frac{n}{2}$.

The following proposition asserts the equivalence of inhomogeneous Sobolev norms and will be used in Section 10 on scattering. The difference from Proposition \ref{prop:h equ} is that Assumption \ref{as:equ}(1) and (3) are not needed.

\begin{proposition}\label{prop:inh:sobolev}
Let $1 < p < \infty$, assume $V \in L^{\infty}$, and let $E > 0$ be sufficiently large. Then for any $0 < s \leq 4$,
$(\Delta^{2} + E)^{s/4}(H + E)^{-s/4}$ and $(H + E)^{s/4}(\Delta^{2}+E)^{-s/4}$ are bounded on $L^{p}$.
\end{proposition}

\begin{proof}
First we show that $(\Delta^{2} + E)(H + E)^{-1} \in \mathbb{B}(L^{p})$. Since $H$ is bounded from below, there exists $E_{0} > 0$ such that for $E \geq E_{0}$, $H + E$ is a positive self-adjoint operator. In this case we may write
\[ (H + E)^{-1}f = \int_{0}^{\infty} e^{-tH}e^{-Et}fdt,\ \ f \in L^{2}. \]
By \cite{Deng-1}, the semigroup $e^{-tH}$ extends to an analytic semigroup $e^{zH}$ on $L^{1}$ with angle $\frac{\pi}{2}$, and its kernel satisfies
\begin{equation}
|e^{-tH}(x, y)| \lesssim t^{-\frac{n}{4}} \exp \left \{ -c\dfrac{|x - y|^{\frac{4}{3}}}{t^{\frac{1}{3}}} + \omega t \right\},\ \ t > 0,
\end{equation}
where $c, \omega > 0$. For any $1 \leq p \leq \infty$ and $t \geq 0$,
\begin{align*}
\int_{\mathbb{R}^{n}} t^{-\frac{n}{4}} \exp \left \{ -c\dfrac{|x - y|^{\frac{4}{3}}}{t^{\frac{1}{3}}} + \omega t \right\} dy \lesssim e^{\omega t}\int_{0}^{\infty} e^{-s}s^{\frac{3}{4}n - 1} ds \leq Ce^{\omega t}\Gamma\left(\frac{3}{4}n\right).
\end{align*}
Hence
\[ \|e^{-tH}\|_{L^{p} \rightarrow L^{p}} \lesssim e^{\omega t } \]
follows. Therefore, if $E > \max(E_{0}, \omega)$, then for $f \in L^{2} \cap L^{p}$,
\[ \|(H + E)^{-1}f\|_{L^{p}} \leq \int_{0}^{\infty}e^{-Et}\|e^{-tH}f\|_{L^{p}}dt \lesssim \int_{0}^{\infty}e^{-(E-\omega )t}dt \|f\|_{L^{p}} \lesssim |E - \omega|^{-1}\|f\|_{L^{p}}. \]
Thus $(H + E)^{-1}$ extends to a bounded operator on $L^{p}$. Moreover,
\[ \Delta^{2}(H + E)^{-1}f = (H + E -V -E)(H + E)^{-1}f = \{1 - (V + E)(H + E)^{-1}\}f \]
so that $\|\Delta^{2}(H + E)^{-1}f\|_{L^{p}} \lesssim (1 + \|V\|_{L^{\infty}})\|f\|_{L^{p}}$ for $f \in L^{2}\cap L^{p}$. By density, we obtain $(\Delta^{2} + E)(H + E)^{-1} \in \mathbb{B}(L^{p})$. The case $0 < s < 4$ follows from Stein's interpolation theorem exactly as in Proposition \ref{prop:h equ}.

Next we show $(H + E)(\Delta^{2}+E)^{-1}$. Similarly, $(\Delta^{2} + E)^{-1}$ is bounded on $L^{p}$. Furthermore,
\[ (H + E)(\Delta^{2} + E)^{-1}f = \{1 + V(\Delta^{2} + E)^{-1}\}f \]
so that $\|(H + E)(\Delta^{2} + E)^{-1}f\|_{L^{p}} \lesssim (1 + \|V\|_{L^{\infty}})\|f\|_{L^{p}}$ for $f \in L^{2}\cap L^{p}$. Hence $(H + E)(\Delta^{2} + E)^{-1} \in \mathbb{B}(L^{p})$. The case $0 < s < 4$ is the same.
\end{proof}

\section{Strichartz Estimates and $L^{p}$-$L^{q}$ Decay Estimates}
For an external force term $h\in L_{loc}^{1}(I, H^{-4})$ and initial data $u_{0} \in L^{2}$, define
\begin{equation}\label{eq:duhamel}
u(t) = e^{itH}u_{0} + i\int_{0}^{t} e^{i(t - s)H}h(s) ds.
\end{equation}

We say that a pair $(q, r)$ is Schr\"{o}dinger admissible (abbreviated S-admissible) if $2 \leq q, r \leq \infty$, $(q, r, n) \neq (2, \infty, 2)$, and
\[ \dfrac{2}{q} + \dfrac{n}{r} = \dfrac{n}{2}. \]
We say that a pair $(q, r)$ is biharmonic admissible (abbreviated B-admissible) if $2 \leq q, r \leq \infty$, $(q, r, n) \neq (2, \infty, 4)$, and
\[ \dfrac{4}{q} + \dfrac{n}{r} = \dfrac{n}{2}. \]

In order to state the assumptions on $V$, define
\begin{equation}
\begin{array}{lcr}
H_{\ell} = 4^{\ell}\Delta^{2}, & V_{\ell}(x) = (-x\cdot \nabla_{x})^{\ell}V(x), & \ell = 1, 2.
\end{array}
\end{equation}
\begin{assumption}\label{as:strichartz}
Let $V$ be a real-valued potential satisfying $V \in L^{\frac{n}{4}, \infty}$. Assume that for each $\ell = 0, 1, 2$, we have $(x\cdot \nabla)^{\ell}V \in L^{1}_{loc}$ and $|(x\cdot \nabla)^{\ell}V|^{1/2}(\Delta^{2} + 1)^{-1/2}$ is bounded on $L^{2}$. Furthermore, assume that for $u \in C_{0}^{\infty}$,
\begin{align}
  \langle (\Delta^{2} + V)u, u \rangle &\gtrsim \langle \Delta^{2}u, u \rangle, \\
  \langle (H_{1} + V_{1})u, u \rangle &\gtrsim \langle \Delta^{2}u, u \rangle, \\
  |\langle (H_{2} + V_{2})u, u \rangle| &\lesssim \langle (H_{1} + V_{1})u, u \rangle.
\end{align}
\end{assumption}

The following Strichartz estimates for \eqref{eq:duhamel} were proved in \cite{Miz}.
\begin{proposition}\label{prop:strichartz}
Let $n \geq 5$, and let $H = \Delta^{2} + V$ satisfy Assumption \ref{as:strichartz}. Then the following statements hold for $u$ given by \eqref{eq:duhamel}.
\begin{itemize}
\item (Standard Strichartz estimate)

Let $(p_{1}, q_{1})$ and $(p_{2}, q_{2})$ be B-admissible. Then
\begin{equation}\label{standard}
\|u\|_{L^{p_{1}}(I, L^{q_{1}})} \lesssim \|u_{0}\|_{L^{2}} + \|h\|_{L^{p_{2}^{\prime}}(I, L^{q_{2}^{\prime}})}.
\end{equation}
\item (Improved Strichartz estimate)

Let $(p, q)$ and $(\tilde{p}, \tilde{q})$ be S-admissible. Then
\begin{equation}\label{improved}
\| |\nabla|^{\frac{2}{p}}u\|_{L^{p}(I, L^{q})} \lesssim \|u_{0}\|_{L^{2}} + \||\nabla|^{-\frac{2}{\tilde{p}}}h\|_{L^{\tilde{p}^{\prime}}(I, L^{\tilde{q}^{\prime}})}.
\end{equation}
\end{itemize}
\end{proposition}

\begin{corollary}\label{sobolev stricartz}
Let $n \geq 5$, and let $H = \Delta^{2} + V$ satisfy Assumptions \ref{as:equ} and \ref{as:strichartz}. Then for $u$ given by \eqref{eq:duhamel} and any B-admissible pair $(q, r)$ satisfying $r < \frac{n}{2}$, we have
\begin{equation}\label{strichartz}
\|\Delta u\|_{L^{q}(I, L^{r})} \lesssim \|\Delta u_{0}\|_{L^{2}} + \|\nabla h\|_{L^{2}(I, L^{\frac{2n}{n + 2}})}.
\end{equation}
\end{corollary}

\begin{proof}
Let $(q, r)$ be any B-admissible pair with $r < \dfrac{n}{2}$. By Sobolev's inequality,
\begin{equation}\label{apply sobolev}
\||\nabla|^{2 - \frac{2}{q}} \int_{0}^{t} e^{i(t - s)H}h(s) ds\|_{L^{q}(I, L^{r})} \lesssim \|\Delta \int_{0}^{t} e^{i(t - s)H}h(s) ds\|_{L^{q}(I, L^{\tilde{r}})}
\end{equation}
where $\frac{1}{\tilde{r}} = \frac{1}{r} + \frac{2}{nq}$. Then $(q, \tilde{r})$ is S-admissible. Hence, by the Strichartz estimate \eqref{improved} and \eqref{eq1}, \eqref{eq2},
\begin{equation}
\begin{aligned}\label{calc-1}
\|\Delta \int_{0}^{t} e^{i(t - s)H}h(s) ds\|_{L^{q}(I, L^{\tilde{r}})} &\lesssim \| H^{\frac{1}{2q}}\int_{0}^{t} e^{i(t - s)H}H^{\frac{1}{2} - \frac{1}{2q}}h(s) ds\|_{L^{q}(I, L^{\tilde{r}})} \\
&\lesssim \||\nabla|^{-1}H^{\frac{1}{2} - \frac{1}{2q}}h\|_{L^{2}(I, L^{\frac{2n}{n + 2}})} \\
&\lesssim \|H^{\frac{1}{4} - \frac{1}{2q}}h\|_{L^{2}(I, L^{\frac{2n}{n + 2}})} \\
&\lesssim \||\nabla|^{1 - \frac{2}{q}} h\|_{L^{2}(I, L^{\frac{2n}{n + 2}})}.
\end{aligned}
\end{equation}
From \eqref{apply sobolev} and \eqref{calc-1},
\begin{equation}\label{conc-1}
\||\nabla|^{2 - \frac{2}{q}} \int_{0}^{t} e^{i(t - s)H}h(s) ds\|_{L^{q}(I, L^{r})} \lesssim \||\nabla|^{1 - \frac{2}{q}} h\|_{L^{2}(I, L^{\frac{2n}{n + 2}})}.
\end{equation}
Using \eqref{conc-1}, for any B-admissible pair $(q, r)$,
\begin{equation}
\begin{aligned}\label{inhomogeneous}
\|\Delta \int_{0}^{t} e^{i(t - s)H}h(s) ds\|_{L^{q}(I, L^{r})} &\lesssim \|H^{\frac{1}{2} - \frac{1}{2q}} \int_{0}^{t} e^{i(t - s)H}H^{\frac{1}{2q}}h(s) ds\|_{L^{q}(I, L^{r})} \\
&\lesssim \||\nabla|^{2 - \frac{2}{q}} \int_{0}^{t} e^{i(t - s)H}H^{\frac{1}{2q}}h(s) ds\|_{L^{q}(I, L^{r})} \\
&\lesssim \||\nabla|^{1 - \frac{2}{q}} H^{\frac{1}{2q}}h\|_{L^{2}(I, L^{\frac{2n}{n + 2}})} \\
&\lesssim \|H^{\frac{1}{4}}h\|_{L^{2}(I, L^{\frac{2n}{n + 2}})}  \\
&\lesssim \|\nabla h\|_{L^{2}(I, L^{\frac{2n}{n + 2}})}.
\end{aligned}
\end{equation}
Also,
\begin{equation}\label{homogeneous}
\|\Delta e^{itH}u_{0}\|_{L^{q}(I, L^{r})} \lesssim \|e^{itH}H^{1/2}u_{0}\|_{L^{q}(I, L^{r})} \lesssim \|H^{1/2}u_{0}\|_{L^{2}} \lesssim \|\Delta u_{0}\|_{L^{2}}.
\end{equation}
Combining \eqref{inhomogeneous} and \eqref{homogeneous}, we obtain \eqref{strichartz}.
\end{proof}

\begin{remark}
The pair $\left(\frac{2(n+4)}{n-4}, \frac{2n(n+4)}{n^{2}+16} \right)$ is B-admissible, and for $n \geq 5$ we have $\frac{2n(n+4)}{n^{2}+16} < \frac{n}{2}$. Therefore, by Sobolev's inequality and the Strichartz estimate \eqref{strichartz}, for $u$ given by \eqref{eq:duhamel},
\[ \|u\|_{W(I)} \lesssim \|u\|_{M(I)} \lesssim \|u_{0}\|_{\dot{H}^{2}} + \|h\|_{N(I)} \]
holds. This inequality will be used in Sections 6 and later.
\end{remark}

The following estimate will be needed in the discussion of scattering in Section 10.
\begin{proposition}
Assume that $V$ satisfies Assumptions \ref{as:equ} and \ref{as:strichartz}. For $h\in L^{2}(\mathbb{R}, \dot{H}^{1, \frac{2n}{n + 2}})$,
\begin{equation}\label{strichartz for scattering}
\|\int_{\mathbb{R}} e^{-isH}h(s)ds\|_{\dot{H}^{2}} \lesssim \|\nabla h\|_{L^{2}(\mathbb{R}, L^{\frac{2n}{n + 2}})}.
\end{equation}
\begin{proof}
\begin{align*}
\|\int_{\mathbb{R}} e^{-isH}h(s)ds\|_{L^{2}}^{2} &= \left| \int_{\mathbb{R}} \int_{\mathbb{R}} \langle e^{i(t - s)H}h(s), h(t) \rangle ds dt\right| \\
&= \left| \int_{\mathbb{R}} \int_{\mathbb{R}} \langle \nabla e^{i(t - s)H}h(s), \nabla^{-1} h(t) \rangle ds dt\right| \\
&\leq \|\nabla \int_{\mathbb{R}} e^{-i(t-s)H}h(s)ds\|_{L^{2}(\mathbb{R}, L^{\frac{2n}{n - 2}})}\|\nabla^{-1}h\|_{L^{2}(\mathbb{R}, L^{\frac{2n}{n + 2}})} \\
&\lesssim \|\nabla^{-1}h\|_{L^{2}(\mathbb{R}, L^{\frac{2n}{n + 2}})}^{2}.
\end{align*}
Hence, by \eqref{eq1} and \eqref{eq2},
\begin{align*}
\|\Delta \int_{\mathbb{R}} e^{-isH}h(s)ds\|_{L^{2}} &\lesssim \|\int_{\mathbb{R}} e^{-isH}H^{1/2}h(s)ds\|_{L^{2}} \\
&\lesssim \|\nabla^{-1}H^{1/2}h\|_{L^{2}(\mathbb{R}, L^{\frac{2n}{n + 2}})} \\
&\lesssim \|H^{1/4}h\|_{L^{2}(\mathbb{R}, L^{\frac{2n}{n + 2}})} \\
&\lesssim \|\nabla h\|_{L^{2}(\mathbb{R}, L^{\frac{2n}{n + 2}})}.
\end{align*}
\end{proof}
\end{proposition}

Next we consider $L^{p}$-$L^{q}$ decay estimates. Define
\[ I(t, x) = \dfrac{1}{(2\pi)^{n}}\int_{\mathbb{R}^{n}} e^{it|\xi|^{4} - i\langle x, \xi \rangle} d\xi.\]
Then $I$ is the fundamental solution of \eqref{eq:duhamel} in the case $V = 0$, i.e., $H = \Delta^{2}$. By \cite{Ben}, the following estimate is known:

\begin{equation}\label{eq:dispersion}
|D^{\alpha}I(t, x)| \leq Ct^{-\frac{n + |\alpha|}{4}}(1 + t^{-\frac{1}{4}}|x|)^{\frac{|\alpha|- n}{3}}.
\end{equation}
From \eqref{eq:dispersion} one obtains the time decay estimate
\begin{equation}\label{lplq}
\|e^{it\Delta^{2}}g\|_{L^{p}} \lesssim |t|^{-\frac{n}{4}(1 - \frac{2}{p})}\|g\|_{L^{p^{\prime}}}
\end{equation}
for $2 \leq p \leq \infty$, where $p^{\prime}$ denotes the conjugate exponent of $p$. We would like to obtain an analogous decay estimate in the case $H = \Delta^{2} + V$. To this end, we introduce wave operators.

By \cite{Burak}, the following boundedness of wave operators on $L^{p}$ is known.
\begin{proposition}\label{wave bounded}
Assume that $V$ satisfies Assumption \ref{as:wave}. Then the wave operators
\begin{equation}
W_{\pm} := s - \lim_{t \rightarrow \pm \infty} e^{itH}e^{-it\Delta^{2}}
\end{equation}
are bounded on $L^{p}$ for every $1 < p < \infty$.
\end{proposition}

\begin{proposition}\label{prop:lp-lq}
Assume that $V$ satisfies Assumption \ref{as:wave}. Then for $H = \Delta^{2} + V$, the $L^{p}$-$L^{q}$ estimate
\begin{equation}\label{lplq for H}
\|e^{itH}P_{ac}(H)g\|_{L^{p}} \lesssim |t|^{-\frac{n}{4}(1 - \frac{2}{p})}\|g\|_{L^{p^{\prime}}}
\end{equation}
holds for $2 \leq p < \infty$, where $p^{\prime}$ is the conjugate exponent of $p$.
\begin{proof}
It is known that for any Borel measurable function $f$, the wave operators $W_{\pm}$ satisfy
\begin{equation}\label{base:wave}
f(H)W_{\pm} = W_{\pm}f(\Delta^{2}).
\end{equation}
Also, since $W_{\pm}W_{\pm}^{*} = P_{ac}(H)$, combining this with \eqref{base:wave}, we obtain
\begin{equation}\label{base:wave2}
f(H)P_{ac}(H) = W_{\pm}f(\Delta^{2})W_{\pm}^{*}.
\end{equation}
By Proposition \ref{wave bounded}, \eqref{lplq}, and \eqref{base:wave2},
\[ \|e^{itH}P_{ac}(H)\|_{L^{p^{\prime}}\rightarrow L^{p}} \leq \|W_{\pm}\|_{L^{p}\rightarrow L^{p}}\|e^{it\Delta^{2}}\|_{L^{p^{\prime}}\rightarrow L^{p}}\|W_{\pm}^{*}\|_{L^{p^{\prime}}\rightarrow L^{p^{\prime}}} \lesssim |t|^{-\frac{n}{4}(1 - \frac{2}{p})}. \]
\end{proof}
\end{proposition}

\begin{remark}\label{remark:st:wave:as}
Suppose that $V$ satisfies Assumption \ref{as:strichartz}, that is, the Strichartz estimates hold. Let $u(t) = e^{-itH}f$. Then for any $w \in L^{\frac{n}{2}}$, the endpoint Strichartz estimate yields
\begin{align*}
\|wu\|_{L^{2}(I, L^{2})} \leq \|w\|_{L^{\frac{n}{2}}}\|u\|_{L^{2}(I,L^{\frac{2n}{n - 4}})}\lesssim \|f\|_{L^{2}}.
\end{align*}
It is known that this is equivalent to
\begin{equation}\label{hac}
\sup_{z \in \mathbb{C} \backslash \mathbb{R}}\|w\mathrm{Im}(H - z)^{-1}w\|_{L^{2}\rightarrow L^{2}} < \infty.
\end{equation}
From \eqref{hac}, it is also known that $\mathcal{H}_{ac}(H) = L^{2}$. Hence in this case $P_{ac}(H)$ is the identity operator. Also, if we set $u(t) = \int_{0}^{t}e^{-i(t-s)H}F(s) ds$, then for $w_{1},w_{2}\in L^{\frac{n}{2}}$,
\begin{equation}\label{hac-2}
\begin{aligned}
\|w_{1}u\|_{L^{2}(I, L^{2})} &\leq \|w_{1}\|_{L^{\frac{n}{2}}}\|u\|_{L^{2}(I, L^{\frac{2n}{n-4}})}\\
&\lesssim \|w_{1}\|_{L^{\frac{n}{2}}}\|F\|_{L^{2}(I, L^{\frac{2n}{n+4}})}\\ &\leq \|w_{1}\|_{L^{\frac{n}{2}}}\|w_{2}\|_{L^{\frac{n}{2}}}\|w_{2}^{-1}F\|_{L^{2}(I, L^{2})}.
\end{aligned}
\end{equation}
From \eqref{hac-2},
\begin{equation}
\sup_{z \in \mathbb{C} \backslash \mathbb{R}}\|w_{1}(H - z)^{-1}w_{2}\|_{L^{2}\rightarrow L^{2}} < \infty
\end{equation}
follows. It is known that this implies that zero energy is regular for $H$.
\end{remark}

\begin{remark}
Under Assumption \ref{as:equ}, if the boundedness of wave operators is available, that is, if $V$ satisfies Assumptions \ref{as:wave} and \ref{as:equ}, then one can derive the Strichartz estimate \eqref{strichartz}. Indeed, under Assumptions \ref{as:wave} and \ref{as:equ}, $P_{ac}(H)$ is the identity. Writing $H_{0} = \Delta^{2}$, by the equivalence of Sobolev norms, the identity $e^{itH} = W_{\pm}e^{itH_{0}}W_{\pm}^{*}$, and the boundedness of wave operators, we have
\begin{align*}
\|\Delta e^{itH}u_{0}\|_{L^{p}(I, L^{q})} &\lesssim \|H^{\frac{1}{2}}W_{\pm}e^{itH_{0}}W_{\pm}^{*}u_{0}\|_{L^{p}(I, L^{q})} \\
&\lesssim \|W_{\pm}H_{0}^{\frac{1}{2}}e^{itH_{0}}W_{\pm}^{*}u_{0}\|_{L^{p}(I, L^{q})} \\
&\lesssim \|\Delta e^{itH_{0}}W_{\pm}^{*}u_{0}\|_{L^{p}(I, L^{q})} \\
&\lesssim \|\Delta W_{\pm}^{*}u_{0}\|_{L^{2}} \lesssim \|H^{\frac{1}{2}}  W_{\pm}^{*} u_{0}\|_{L^{2}} \lesssim \|\Delta u_{0}\|_{L^{2}}
\end{align*}
for any B-admissible pair $(p, q)$ with $q < \frac{n}{2}$. Furthermore,
\begin{align*}
\|\Delta \int_{0}^{t}e^{i(t-s)H}h(s)ds\|_{L^{p}(I, L^{q})} &\lesssim \|H^{\frac{1}{2}}W_{\pm}\int_{0}^{t}e^{i(t-s)H_{0}}W_{\pm}^{*}h(s)ds\|_{L^{p}(I, L^{q})} \\
&\lesssim  \|W_{\pm}H^{\frac{1}{2}}_{0}\int_{0}^{t}e^{i(t-s)H_{0}}W_{\pm}^{*}h(s)ds\|_{L^{p}(I, L^{q})} \\
&\lesssim \|\Delta \int_{0}^{t}e^{i(t-s)H_{0}}W_{\pm}^{*}h(s)ds\|_{L^{p}(I, L^{q})}\\
&\lesssim \|\nabla W_{\pm}^{*}h\|_{L^{2}(I, L^{\frac{2n}{n+2}})} \lesssim \|\nabla h\|_{L^{2}(I, L^{\frac{2n}{n+2}})}
\end{align*}
so \eqref{strichartz} follows. Therefore, in what follows, when the boundedness of wave operators is assumed, we use the Strichartz estimate \eqref{strichartz} without explicitly assuming Assumption \ref{as:strichartz}.
\end{remark}

\section{Local and Global Existence in the Subcritical Case}

\begin{proposition}
Assume that $V \in L^{\infty}$ satisfies Assumption \ref{as:equ} and Assumption \ref{as:strichartz}. Let $u_{0} \in H^{2}$, and let $n \geq 5$ and $p \in (1, 2^{\sharp} -1)$. Then, for any such $u_{0}$ and $p$, there exist $T > 0$ and a solution $u \in C([0, T], H^{2})$ to \eqref{eq:NL4S} satisfying $u(0)=u_{0}$. Moreover, the conservation of mass and energy,
\begin{equation}\label{mass energy con}
\begin{aligned}
M(u(t)) &= M(u_{0}), \\
E(u(t)) &= E(u_{0}),
\end{aligned}
\end{equation}
hold for all $t \in [0, T]$.
\end{proposition}

\begin{proof}
Let $M = C\|u_{0}\|_{\dot{H}^{2}}$, and define the mapping $\Phi$ by
\[
\Phi(u) = e^{itH}u_{0} - i\int_{0}^{t}e^{i(t-s)H}|u|^{p-1}u(s)\, ds.
\]
Next, define the complete metric space $X_{M}$ by
\[
X_{M} = \left\{ u \in C(I, H^{2}) ;\ \|\Delta u \|_{S(I)} \leq 2M \right\}.
\]
Here,
\[
\|u\|_{S(I)} := \sup_{\substack{(q, r):\text{B-admissible} \\ r < n/2}}\|u\|_{L^{q}(I,L^{r})},
\]
and the distance on $X_M$ is defined by this norm.
By the Strichartz estimate \eqref{strichartz}, together with H\"{o}lder's inequality and Sobolev's inequality, we have for any $u \in X_{M}$,
\begin{equation}\label{exist:sub:1}
\begin{aligned}
\|\Delta \Phi(u) \|_{S(I)}
&\leq C\|u_{0}\|_{\dot{H}^{2}} + C\|u\|^{p - 1}_{L^{2(p-1)}(I, L^{\frac{n(p-1)}{2}})}\| \nabla  u\|_{L^{\infty}(I, L^{\frac{2n}{n-2}})} \\
&\leq C\|u_{0}\|_{\dot{H}^{2}} + C\|u\|^{p - 1}_{L^{2(p-1)}(I, L^{\frac{n(p-1)}{2}})}\| \Delta  u\|_{L^{\infty}(I, L^{2})}.
\end{aligned}
\end{equation}
Also, by Sobolev's inequality and H\"{o}lder's inequality, for $1 < p < 2^{\sharp} - 1$, we obtain
\begin{equation}\label{exist:sub:3}
\|u\|_{L^{2(p-1)}(I, L^{\frac{n(p-1)}{2}})} \leq C T^{\frac{1}{(p-1)} - \frac{n - 4}{8}} \|\Delta u \|_{L^{\frac{8(p-1)}{(n-4)(p-1) - 4}}(I, L^{\frac{n(p-1)}{2p}})}.
\end{equation}
From \eqref{exist:sub:1} and \eqref{exist:sub:3}, setting
\[
\theta = \frac{1}{(p-1)} - \frac{n - 4}{8} > 0,
\]
we obtain
\[
\| \Delta  \Phi(u) \|_{S(I)} \leq C\|u_{0}\|_{\dot{H}^{2}} + CT^{\theta}\| \Delta  u\|^{p}_{S(I)} \leq M + CT^{\theta}M^{p}.
\]
Therefore, if we choose $T$ sufficiently small so that $CT^{\theta}M^{p} \leq M$, namely,
\[
T \leq CM^{\frac{1-p}{\theta}},
\]
then $\Phi$ maps $X_{M}$ into itself.
Furthermore, for any $u, v \in X_{M}$, the following holds:
\begin{align*}
\| \Phi(u) - \Phi(v) \|_{S(I)}
&\lesssim \|(|u|^{p-1} + |v|^{p-1})|u-v|\|_{L^{2}(I, L^{\frac{2n}{n+4}})} \\
&\lesssim \left(\|u\|^{p - 1}_{L^{2(p-1)}(I, L^{\frac{n(p-1)}{2}})} + \|v\|^{p - 1}_{L^{2(p-1)}(I, L^{\frac{n(p-1)}{2}})}\right)\| u -v \|_{L^{\infty}(I, L^{2})}\\
&\leq CT^{\theta}\left( \| \Delta u\|^{p - 1}_{S(I)} + \|\Delta  v\|^{p - 1}_{S(I)} \right) \|u - v\|_{S(I)}\\
&\leq CT^{\theta}M^{p-1} \|u - v\|_{S(I)}.
\end{align*}
Thus, by taking $T$ sufficiently small, $\Phi$ becomes a contraction mapping on $X_M$.
The conservation of mass and energy follows from Chapter 3 of \cite{Caz}.
\end{proof}

\begin{corollary}
Let $n \geq 5$, $p \in (1, 2^{\sharp} -1)$, and assume that $V \geq 0$ and $u_{0} \in H^{2}$. Let $u$ be the solution to \eqref{eq:NL4S} with initial data $u_{0}$.
Then the solution $u$ can be extended globally in time on $\mathbb{R}$ in each of the following cases:
\begin{enumerate}
\item[(a)] \ \ $\lambda \geq 0$,
\item[(b)] \ \ $\lambda < 0,\ p < 1 + \frac{8}{n}$,
\item[(c)] \ \ $\lambda < 0,\ p = 1 + \frac{8}{n}$ and the $L^{2}$-norm of $u_{0}$ is sufficiently small,
\item[(d)] \ \ $\lambda < 0$ and the $H^{2}$-norm of $u_{0}$ is sufficiently small.
\end{enumerate}
\end{corollary}

\begin{proof}
It suffices to show that $\|u(t)\|_{\dot{H}^{2}}$ is uniformly bounded in $t$.
Case (a) follows easily from the conservation of energy.
Cases (b) and (c) follow from the Gagliardo--Nirenberg inequality together with the conservation of mass and energy.
Indeed,
\begin{equation}\label{sub:global:1}
\begin{aligned}
\int_{\mathbb{R}^{n}}|\Delta u(t, x) |^{2} \, dx
&\leq E(u_{0}) - \int_{\mathbb{R}^{n}} V(x)|u(x)|^{2}\,dx - \dfrac{2\lambda}{p + 1}\int_{\mathbb{R}^{n}}|u(t, x)|^{p + 1}\,dx \\
&\leq E(u_{0}) + C\|u_{0}\|_{L^{2}}^{p + 1 - \frac{n(p-1)}{4}}\|\Delta u\|_{L^{2}}^{\frac{n(p-1)}{4}}
\end{aligned}
\end{equation}
holds. Hence, by applying Appendix \ref{bootstrap}, we obtain the desired uniform boundedness in cases (b) and (c).
Finally, we prove (d).
If $\|u_{0}\|_{H^{2}}$ is taken sufficiently small, then we see that $E(u_{0}) > 0$.
Moreover, since $p > 1 + \frac{8}{n}$, we have
\[
\frac{n(p-1)}{4} > 2.
\]
Therefore, it follows from \eqref{sub:global:1} that the boundedness is obtained by taking $\|u_{0}\|_{L^{2}}$ sufficiently small.
\end{proof}

\section{Local existence in the critical case}

In this section, we consider local solutions to \eqref{eq:NL4S} in the energy-critical case, namely when $p = 2^{\sharp} - 1$.
If $u\in C(I, H^{2})$ is a solution to \eqref{eq:NL4S} in the energy-critical case with initial data $u_{0}$, then it is represented as

\begin{equation}\label{duhamel2}
u(t) = e^{itH}u_{0} + i\lambda \int_{0}^{t} e^{i(t - s)H}|u(s)|^{\frac{8}{n - 4}}u(s) ds.
\end{equation}

The following uniqueness statement can be obtained in the same way as in Chapter 4 of \cite{Caz}.

\begin{proposition}[Uniqueness in $H^{2}$]
Let $n \geq 5$, and assume that $V$ satisfies Assumption \ref{as:strichartz}. Let $u_{0} \in H^{2}$. Suppose that $u_{1}, u_{2} \in C(I, H^{2})$ are two solutions to \eqref{eq:NL4S} with the same initial data $u_{0}$. Then $u_{1} = u_{2}$.
\end{proposition}

\begin{proposition}\label{local existence}
Assume that $V$ satisfies Assumptions \ref{as:equ} and \ref{as:strichartz}. Let $n \geq 5$, and let $u_{0} \in H^{2}$ be arbitrary. For any interval $I = [0, T]$, there exists $\delta > 0$ such that, if
\begin{equation}\label{small}
\|e^{itH}u_{0}\|_{W(I)} < \delta,
\end{equation}
then there exists a unique solution $u\in C(I, H^{2})$ to \eqref{eq:NL4S} with initial data $u_{0}$, and moreover
$u\in M(I) \cap L^{\frac{2(n + 4)}{n}}(I \times \mathbb{R}^{n})$.
In addition, the mass and energy are conserved in the sense of \eqref{mass energy con}. Furthermore,
\begin{equation}\label{st}
\begin{array}{l}
\|u\|_{W(I)} < 2 \delta, \\
\|u\|_{M(I)} + \|u\|_{L^{\infty}(I, H^{2})} \leq C\left(\|u_{0}\|_{H^{2}} + \delta^{\frac{n + 4}{n - 4}}\right)
\end{array}
\end{equation}
holds for some constant $C > 0$ independent of $u_{0}$.
Moreover, depending on $\delta$, there exists $\delta_{0}$ such that for any $\delta_{1} \in (0, \delta_{0})$, if $\|u_{0} - v_{0}\|_{H^{2}} < \delta_{1}$ and $v$ denotes the solution to \eqref{eq:NL4S} with initial data $v_{0}$, then $v$ is defined on $I$ and, for any B-admissible pair $(q, r)$,
\begin{equation}\label{initconti}
\|u - v\|_{L^{q}(I, L^{r})} \leq C\delta_{1}
\end{equation}
holds, where $C > 0$ is independent of $u_{0}, v_{0}$.
\end{proposition}

\begin{proof}
For $u \in W(I)$, let $\Phi_{u_{0}}(u)$ denote the right-hand side of \eqref{duhamel2}. Let $M = C\|u_{0}\|_{L^{2}}$, and for sufficiently small $\delta > 0$, define
\[
X_{M, \delta} = \left\{ v \in M(I) ; \|v\|_{W(I)} < 2\delta,\ \|v\|_{L^{\frac{2(n + 4)}{n}}(I, L^{\frac{2(n + 4)}{n}})} \leq 2M \right\}.
\]
We equip $X_{M, \delta}$ with the metric induced by the norm
$L^{\frac{2(n + 4)}{n}}(I, L^{\frac{2(n + 4)}{n}})$.
We show that $\Phi_{u_{0}}$ is a contraction mapping on $X_{M, \delta}$.

Using the Strichartz estimates \eqref{standard} and \eqref{strichartz}, we obtain
\begin{align*}
\|\Phi_{u_{0}}(u)\|_{W(I)}
&\leq \|e^{itH}u_{0}\|_{W(I)} + C\|\nabla h\|_{L^{2}(I, L^{\frac{2n}{n + 2}})} \\
&\leq \delta + C\|u\|_{W(I)}^{\frac{n+4}{n-4}} \\
&\leq \delta + C\delta^{\frac{n+4}{n-4}} \leq 2\delta,
\end{align*}
and
\begin{align*}
\|\Phi_{u_{0}}(u)\|_{L^{\frac{2(n + 4)}{n}}(I, L^{\frac{2(n + 4)}{n}})}
&\lesssim \|u_{0}\|_{L^{2}} + \|u\|_{Z(I)}^{\frac{8}{n-4}} \|u\|_{L^{\frac{2(n + 4)}{n}}(I, L^{\frac{2(n + 4)}{n}})} \\
&\leq M + \delta^{\frac{8}{n - 4}}M \leq 2M.
\end{align*}
Hence $\Phi_{u_{0}} : X_{M, \delta} \to X_{M, \delta}$.

Next, we show that $\Phi_{u_{0}}$ is a contraction:
\begin{align*}
\|\Phi_{u_{0}}(u) - \Phi_{u_{0}}(v)\|_{L^{\frac{2(n + 4)}{n}}(I, L^{\frac{2(n + 4)}{n}})}
&\lesssim \| |u|^{\frac{8}{n-4}}u - |v|^{\frac{8}{n-4}}v \|_{L^{\frac{2(n + 4)}{n + 8}}(I, L^{\frac{2(n + 4)}{n + 8}})} \\
&\lesssim \| (|u|^{\frac{8}{n-4}} + |v|^{\frac{8}{n-4}} ) | u - v | \|_{L^{\frac{2(n + 4)}{n + 8}}(I, L^{\frac{2(n + 4)}{n + 8}})} \\
&\lesssim (\|u\|_{Z(I)}^{\frac{8}{n-4}} + \|v\|_{Z(I)}^{\frac{8}{n-4}})\|u - v \|_{L^{\frac{2(n + 4)}{n}}(I, L^{\frac{2(n + 4)}{n}})} \\
&\lesssim \delta^{\frac{8}{n-4}}\|u - v\|_{L^{\frac{2(n + 4)}{n}}(I, L^{\frac{2(n + 4)}{n}})}.
\end{align*}
Therefore, by the contraction mapping principle, there exists a unique solution $u$ in $X_{M,\delta}$.

We next prove \eqref{initconti}. Let $(q, r)$ be any B-admissible pair. By the Strichartz estimate \eqref{standard},
\begin{equation}\label{eq:6-1}
\|u - v\|_{L^{q}(I, L^{r})} \lesssim \|u_{0} - v_{0}\|_{H^{2}} + C\delta^{\frac{8}{n-4}}\|u - v\|_{L^{\frac{2(n + 4)}{n}}(I, L^{\frac{2(n + 4)}{n}})}.
\end{equation}
Taking $(q, r) = \left(2(n+4)/n, 2(n+4)/n\right)$, we obtain
\[
\|u - v\|_{L^{\frac{2(n + 4)}{n}}(I, L^{\frac{2(n + 4)}{n}})} \leq C\|u_{0} - v_{0}\|_{H^{2}}.
\]
Substituting this into \eqref{eq:6-1} yields \eqref{initconti}.
Finally, \eqref{st} follows by applying the Strichartz estimates once more.
The conservation of mass and energy follows from \cite{Caz}.
\end{proof}

\begin{corollary}\label{global solution for small energy}
Let $n \geq 5$ and $p = 2^{\sharp} - 1$, and assume that $V$ satisfies Assumptions \ref{as:equ} and \ref{as:strichartz}. Then there exists $\varepsilon_{0} > 0$ such that, for any $u_{0} \in H^{2}$ satisfying $\mathcal{E}(u_{0}) \leq \varepsilon_{0}$, there exists a unique global solution $u\in C(\mathbb{R}, H^{2})$ to \eqref{eq:NL4S} with initial data $u_{0}$.
\end{corollary}

\begin{proof}
Assume that $u$ exists on $[0, t_{0}]$ and that the $\dot{H}^{2}$-norm of $u(t_{0})$ is sufficiently small. Then, by the Strichartz estimate \eqref{strichartz}, condition \eqref{small} holds at $t=t_{0}$, and hence $u$ can be extended to $[t_{0}, t_{0}+1]$.

Therefore, to prove global existence, it suffices to show that $\|u\|_{\dot{H}^{2}}$ remains small on the interval of existence.
By Sobolev's inequality and the weak-type H\"older inequality \eqref{weak holder},
\begin{equation}\label{estimate v}
\int_{\mathbb{R}^{n}} V(x) |u(x)|^{2} dx \lesssim \|V\|_{L^{\frac{n}{4}, \infty}}\|u(t)\|_{L^{2^{\sharp}}}^{2} \lesssim \|V\|_{L^{\frac{n}{4}, \infty}}\|u(t)\|_{\dot{H}^{2}}^{2}.
\end{equation}
Moreover, by \eqref{estimate v}, conservation of energy, and Sobolev's inequality,
\begin{equation}
E(u(t)) = E(u_{0}) \leq C(\mathcal{E}(u_{0}) + \mathcal{E}(u_{0})^{2^{\sharp}/2}).
\end{equation}
Since $V \geq 0$, we have
\begin{equation}\label{small energy}
\begin{aligned}
\|u(t)\|^{2}_{\dot{H^{2}}}
&= 2E(u(t)) - \int_{\mathbb{R}^{n}} V(x) |u(x)|^{2} dx - \int_{\mathbb{R}^{n}} \dfrac{n-4}{n}\lambda |u(x)|^{2^{\sharp}} dx\\
&\leq 2E(u_{0}) + |\lambda|\dfrac{n-4}{n}\|u(t)\|_{L^{2^{\sharp}}}^{2^{\sharp}} \\
&\lesssim \mathcal{E}(u_{0}) + \mathcal{E}(u_{0})^{2^{\sharp}/2} + \|u(t)\|_{\dot{H^{2}}}^{2^{\sharp}}.
\end{aligned}
\end{equation}
It follows from \eqref{small energy} that, if $\varepsilon_{0}$ is sufficiently small, then $u$ remains small in the $\dot{H}^{2}$-norm.
\end{proof}

\begin{proposition}\label{criterion}
Assume that $V$ satisfies Assumptions \ref{as:equ} and \ref{as:strichartz}. Let $n \geq 5$ and $p = 2^{\sharp} - 1$.
Let $u\in C([0, T), H^{2})$ be a solution to \eqref{eq:NL4S} satisfying $\|u\|_{Z([0, T])} < +\infty$.
Then there exists a constant $K = K(\|u_{0}\|_{H^{2}}, \|u\|_{Z([0, T])})$ such that
\begin{equation}\label{criterion:con:1}
\|u\|_{L^{\frac{2(n + 4)}{n}}([0, T], L^{\frac{2(n + 4)}{n}})} + \|u\|_{L^{\infty}([0, T], \dot{H^{2}})} + \|u\|_{M([0, T])} \leq K,
\end{equation}
and $u$ can be extended to a solution $\tilde{u}\in C([0, \tilde{T}), H^{2})$ of \eqref{eq:NL4S} for some $\tilde{T} > T$.
\end{proposition}

\begin{proof}
Take a small $\eta > 0$, and let $B = \|u\|_{Z([0, T])}$. For $x \geq 0$, let $[x]$ denote the largest integer not exceeding $x$.
Set
\[
N = \left[ (B/\eta)^{\frac{2(n+4)}{n-4}} \right] + 1.
\]
Then the interval $[0, T]$ can be decomposed into $N$ subintervals $I_{j}$, $j=1,\dots,N$, such that
\[
\|u\|_{Z(I_{j})} = \eta \quad (1 \leq j \leq N - 1), \qquad \|u\|_{Z(I_{N})} \leq \eta.
\]

Applying the Strichartz estimate \eqref{strichartz} on $I_{j} = [t_{j}, t_{j+1}]$, we obtain for any $t \in I_{j}$,
\begin{equation}\label{criterion:0}
\begin{aligned}
\|u\|_{M([t_{j}, t])}
&\leq C\|u(t_{j})\|_{\dot{H}^{2}} + C\|u\|_{Z(I_{j})}^{\frac{8}{n-4}}\|u\|_{M([t_{j}, t])} \\
&\leq C\|u(t_{j})\|_{\dot{H}^{2}} + C\eta^{\frac{8}{n-4}}\|u\|_{M([t_{j}, t])}.
\end{aligned}
\end{equation}
Also, applying \eqref{standard} on $I_{j}$ and using conservation of mass, we get
\begin{equation}\label{criterion:1}
\begin{aligned}
\|u\|_{L^{\frac{2(n+4)}{n}}([t_{j}, t], \ L^{\frac{2(n+4)}{n}})}
&\leq C\|u(t_{j})\|_{L^{2}} + C\|u\|_{Z([t_{j},\ t])}^{\frac{8}{n-4}}\|u\|_{L^{\frac{2(n+4)}{n}}([t_{j}, t], \ L^{\frac{2(n+4)}{n}})} \\
&\leq C\|u_{0}\|_{L^{2}} + C\eta^{\frac{8}{n-4}}\|u\|_{L^{\frac{2(n+4)}{n}}([t_{j}, t], \ L^{\frac{2(n+4)}{n}})}.
\end{aligned}
\end{equation}
If $\eta$ is chosen sufficiently small, \eqref{criterion:0} and \eqref{criterion:1} imply
\begin{align*}
\|u\|_{M(I_{j})} &\leq C\|u(t_{j})\|_{\dot{H}^{2}}, \\
\|u\|_{L^{\frac{2(n+4)}{n}}(I_{j}, \ L^{\frac{2(n+4)}{n}})} &\leq C\|u_{0}\|_{L^{2}}.
\end{align*}
Moreover, by \eqref{strichartz},
\[
\|u\|_{L^{\infty}(I_{j},\ \dot{H}^{2})} \leq C\|u(t_{j})\|_{\dot{H}^{2}},
\]
and in particular,
\[
\|u(t_{j+1})\|_{\dot{H}^{2}} \leq C\|u(t_{j})\|_{\dot{H}^{2}}.
\]
Thus, for every $j$,
\begin{equation}\label{criterion:2}
\begin{aligned}
&\|u\|_{L^{\frac{2(n+4)}{n}}([0, T], \ L^{\frac{2(n+4)}{n}})} \leq N^{\frac{n}{2(n+4)}}C\|u_{0}\|_{L^{2}}, \\
&\|u\|_{L^{\infty}([0,T],\ \dot{H}^{2})} \leq C^{N}\|u_{0}\|_{\dot{H}^{2}}, \\
&\|u\|_{M([0,T])} \leq C^{N}\|u_{0}\|_{\dot{H}^{2}}.
\end{aligned}
\end{equation}
Hence \eqref{criterion:con:1} follows from \eqref{criterion:2}.

Now take $t_{0} \in I_{N}$. By Duhamel's formula and the Strichartz estimate \eqref{strichartz}, for any $t$,
\begin{equation}\label{criterion:3}
\begin{aligned}
\|e^{i(t-t_{0})H}u(t_{0})\|_{W([t_{0}, T])}
&\leq \|u\|_{W([t_{0}, T])} + C\||u|^{\frac{8}{n-4}}u\|_{N([t_{0}, T])} \\
&\leq \|u\|_{W([t_{0}, T])} + C\|u\|_{W([t_{0}, T])}^{\frac{n+4}{n-4}}.
\end{aligned}
\end{equation}
Since the $W([t_{0}, T])$-norm of $u$ is finite, it can be made arbitrarily small by taking $t_{0}\to T$.
Therefore the left-hand side of \eqref{criterion:3} tends to $0$ as $t_{0}\to T$.
Hence, for any $\delta > 0$, there exist $t_{1} \in (0, T)$ and $T' > T$ such that $u(t_{1}) \in H^{2}$ and
\begin{equation}\label{criterion:4}
\|e^{i(t-t_{1})H}u(t_{1})\|_{W([t_{1}, T'])} \leq \delta.
\end{equation}
Then, by Proposition \ref{local existence}, there exists a solution $v \in C([t_{1}, T'], H^{2})$ of \eqref{eq:NL4S} with $v(t_{1}) = u(t_{1})$.
By uniqueness, $u = v$ on $[t_{1}, T)$, and thus $u$ can be extended to $[0, T']$.
\end{proof}

\section{Stability in the critical case}

In this section we discuss the stability of solutions.

\begin{proposition}[Short time perturbation]\label{Short time perturbation}
Assume that $V$ satisfies Assumptions \ref{as:wave} and \ref{as:equ}. Let $n \geq 5$ and $p = 2^{\sharp} - 1$.
Let $I \subset \mathbb{R}$ be a compact time interval containing $0$, and let $\tilde{u}$ be an approximate solution to \eqref{eq:NL4S} on $I \times \mathbb{R}^{n}$, that is, suppose there exists $e \in N(I)$ such that
\begin{equation}\label{approximate NL4S}
i\partial_{t}\tilde{u} + \Delta^{2}\tilde{u} + V\tilde{u} + \lambda|\tilde{u}|^{\frac{8}{n-4}}\tilde{u} = e.
\end{equation}
Assume also the energy bound
\begin{equation}\label{energy bound}
\|\tilde{u}\|_{L^{\infty}(I, \dot{H^{2}})} \leq E
\end{equation}
for some $E >0$.
Let $u_{0} \in H^{2}$ be close to $\tilde{u}(0)$ in the sense that
\begin{equation}\label{close to initial data}
\|u_{0} - \tilde{u}(0) \|_{\dot{H}^{2}} \leq E'
\end{equation}
for some $E' >0$.
Assume further the smallness conditions
\begin{align}
\|\tilde{u}\|_{W(I)} \leq \varepsilon_{0}, \label{small condition 1}\\
\|e^{itH}\left(u_{0} - \tilde{u}(0)\right)\|_{W(I)} \leq \varepsilon, \label{small condition 2}\\
\|e\|_{N(I)} \leq \varepsilon, \label{small condition 3}
\end{align}
for some $0 \leq \varepsilon \leq \varepsilon_{0}$, where $\varepsilon_{0} = \varepsilon_{0}(E, E') > 0$ is sufficiently small.
Then there exists a solution $u\in C(I, H^{2})$ to \eqref{eq:NL4S} with initial data $u_{0}$ such that
\begin{align}
\|u - \tilde{u}\|_{W(I)} &\lesssim \varepsilon + \varepsilon^{\frac{15}{(n-4)^{3}}}, \label{short stability conclusion 1}\\
\|u - \tilde{u}\|_{L^{q}(I, \dot{H}^{2, r})} &\lesssim E' + \varepsilon  + \varepsilon^{\frac{15}{(n-4)^{3}}},\label{short stability conclusion 2}\\
\|u\|_{L^{q}(I, \dot{H}^{2, r})} &\lesssim E + E', \label{short stability conclusion 3}\\
\|(i\partial_{t} + \Delta^{2} + V)(u - \tilde{u}) + e\|_{N(I)} &\lesssim \varepsilon + \varepsilon^{\frac{15}{(n-4)^{3}}}\label{short stability conclusion 4}
\end{align}
for any B-admissible pair $(q, r)$ with $r < \dfrac{n}{2}$.
\end{proposition}

\begin{proof}
We first treat the case $5 \leq n \leq 12$.
Let $v := u - \tilde{u}$ and $f(u) = |u|^{\frac{8}{n - 4}}u$, where $u$ is the solution to \eqref{eq:NL4S} with initial data $u_{0}$ and maximal lifespan $I'$.
Then on $I \cap I'$, equation \eqref{approximate NL4S} implies that $v$ satisfies
\[
\left\{
\begin{aligned}
&i\partial_{t}v + \Delta^{2}v + Vv + f(\tilde{u} + v) - f(\tilde{u}) = e, \\
&v(0, x) = u_{0}(x) - \tilde{u}(0, x).
\end{aligned}
\right.
\]
For $T \in I \cap I'$, define
\[
S(T) := \|f(\tilde{u} + v) - f(\tilde{u}) \|_{N([0, T])}.
\]
By the Strichartz estimate \eqref{strichartz} and \eqref{small condition 2}, \eqref{small condition 3},
\begin{equation}\label{es:w norm of v}
\begin{aligned}
\|v\|_{W([0, T])}
&\lesssim \|e^{itH}v(0)\|_{W([0, T])} + \|f(\tilde{u} + v) - f(\tilde{u}) \|_{N([0, T])} + \|e\|_{N([0, T])} \\
&\lesssim S(T) + \varepsilon.
\end{aligned}
\end{equation}
Hence
\begin{equation}\label{es:z norm of v}
 \|v\|_{Z([0, T])} \lesssim S(T) + \varepsilon.
\end{equation}

On the other hand,
\begin{align*}
\nabla \left[f(\tilde{u} + v) - f(\tilde{u})\right]
= &\, f_{z}(\tilde{u} + v)\nabla(\tilde{u} + v) + f_{\bar{z}}(\tilde{u} + v)\overline{\nabla(\tilde{u} + v)} \\
&-f_{z}(\tilde{u})\nabla\tilde{u} + f_{\bar{z}}(\tilde{u})\overline{\nabla\tilde{u}},
\end{align*}
where
\[
f_{z} := \dfrac{1}{2}\left(\dfrac{\partial f}{\partial x} -i \dfrac{\partial f}{\partial y}\right),\qquad
f_{\bar{z}} := \dfrac{1}{2}\left(\dfrac{\partial f}{\partial x} +i \dfrac{\partial f}{\partial y}\right).
\]
Thus,
\begin{equation}\label{nonlinear with delivative}
\begin{aligned}
\left|\nabla \left[f(\tilde{u} + v) - f(\tilde{u})\right]\right|
\lesssim\ &|\nabla \tilde{u}||v|^{\frac{8}{n-4}} + |\nabla v||\tilde{u} + v|^{\frac{8}{n-4}}.
\end{aligned}
\end{equation}
Therefore, by \eqref{small condition 1}, \eqref{es:w norm of v}, and \eqref{es:z norm of v},
\begin{equation}
\begin{aligned}
S(T)
&\lesssim \| \nabla \tilde{u}|v|^{\frac{8}{n-4}}\|_{N([0,T])}
+ \|\nabla v |\tilde{u}|^{\frac{8}{n-4}}\|_{N([0,T])}
+ \|\nabla v|v|^{\frac{8}{n-4}}\|_{N([0,T])} \\
&\lesssim \|\tilde{u}\|_{W([0,T])}\|v\|_{Z([0,T])}^{\frac{8}{n-4}}
+ \|v\|_{W([0,T])}\|\tilde{u}\|_{Z([0,T])}^{\frac{8}{n-4}}
+ \|v\|_{W([0,T])}^{\frac{n+4}{n-4}} \\
&\lesssim \varepsilon_{0}(S(T) + \varepsilon)^{\frac{8}{n-4}}
+ (S(T) + \varepsilon)\varepsilon_{0}^{\frac{8}{n-4}}
+ (S(T) + \varepsilon)^{\frac{n + 4}{n-4}}.
\end{aligned}
\end{equation}
Since $5 \leq n \leq 12$, we have $\frac{8}{n - 4} \geq 1$. Hence, by Appendix \ref{bootstrap}, if $\varepsilon_{0} = \varepsilon_{0}(E, E')$ is chosen sufficiently small,
\begin{equation}\label{es:S(T)}
S(T) \lesssim \varepsilon
\end{equation}
for any $T \in I \cap I'$.
Thus, by \eqref{es:w norm of v} and \eqref{es:S(T)},
\[
 \|u - \tilde{u}\|_{W(I \cap I')} \lesssim \varepsilon.
\]
By Proposition \ref{criterion}, this implies $I \cap I' = I$, and therefore \eqref{short stability conclusion 1} follows.

Next, let $(q, r)$ be any B-admissible pair with $r < \frac{n}{2}$. By \eqref{strichartz}, \eqref{close to initial data}, and \eqref{small condition 3},
\begin{align*}
\|u - \tilde{u}\|_{L^{q}(I, \dot{H}^{2, r})}
&\lesssim \|u_{0} - \tilde{u}(0)\|_{\dot{H}^{2}} + \|f(\tilde{u} + v) - f(\tilde{u}) \|_{N(I)} + \|e\|_{N(I)} \\
&\lesssim E' + \varepsilon,
\end{align*}
which yields \eqref{short stability conclusion 2}. Finally, to prove \eqref{short stability conclusion 3}, the Strichartz estimate \eqref{strichartz}, together with \eqref{energy bound} and \eqref{close to initial data}, gives
\begin{align*}
\|u \|_{L^{q}(I, \dot{H}^{2, r})}
&\lesssim \|u_{0} - \tilde{u}(0)\|_{\dot{H}^{2}} + \|\tilde{u}(0)\|_{\dot{H}^{2}} + \|u\|_{W(I)}^{\frac{n + 4}{n - 4}} \\
&\lesssim E' + \|\tilde{u}\|_{L^{\infty}(I, \dot{H^{2}})} + \varepsilon_{0}^{\frac{n + 4}{n - 4}}
\lesssim E' + E.
\end{align*}

Now suppose $n > 12$. Define the norms
\begin{align*}
\|u\|_{X(I)} &= \||\nabla|^{\frac{8n}{n^{2} - 16}}u\|_{L^{\frac{n^{2} - 16}{8}}(I, L^{\frac{2(n+4)}{n}})}, \\
\|u\|_{Y(I)} &= \||\nabla|^{\frac{8n}{n^{2} - 16}}u\|_{L^{\frac{n^{2} - 16}{4(n-2)}}(I, L^{\frac{2(n+4)}{n + 8}})}.
\end{align*}

\begin{lemma}[Exotic Strichartz estimate]
For any $F \in Y(I)$,
\begin{equation}\label{exotic strichartz}
\|\int_{0}^{t}e^{i(t - s)H}F(s) ds \|_{X(I)} \lesssim \|F\|_{Y(I)}.
\end{equation}
\begin{proof}
By the $L^{p}$-$L^{q}$ estimate \eqref{lplq for H} and the Hardy--Littlewood--Sobolev inequality,
\begin{equation}\label{exotic calc}
\begin{aligned}
\|\int_{0}^{t}e^{i(t - s)H}F(s) ds \|_{{L^{\frac{n^{2} - 16}{8}}(I, L^{\frac{2(n+4)}{n}})}}
&\lesssim \|\int_{0}^{t}|t - s|^{-\frac{n}{n+4}}\|F\|_{L^{\frac{2(n+4)}{n+8}}_{x}}ds \|_{L^{\frac{n^{2}- 16}{8}}_{t}} \\
&\lesssim \|F\|_{L^{\frac{n^{2} - 16}{4(n-2)}}(I, L^{\frac{2(n+4)}{n + 8}})}.
\end{aligned}
\end{equation}
Since $\frac{2(n+4)}{n} < \frac{n}{2}$ for $n \geq 13$, it follows from \eqref{exotic calc}, \eqref{eq1}, and \eqref{eq2} that
\begin{align*}
\|\int_{0}^{t}e^{i(t - s)H}F(s) ds \|_{X(I)}
&\lesssim \|\int_{0}^{t}e^{i(t - s)H}H^{\frac{2n}{n^{2} - 16}}F(s) ds \|_{{L^{\frac{n^{2} - 16}{8}}(I, L^{\frac{2(n+4)}{n}})}} \\
&\lesssim \|H^{\frac{2n}{n^{2} - 16}}F\|_{L^{\frac{n^{2} - 16}{4(n-2)}}(I, L^{\frac{2(n+4)}{n + 8}})} \\
&\lesssim \|F\|_{Y(I)}.
\end{align*}
\end{proof}
\end{lemma}

We continue the proof of Proposition \ref{Short time perturbation}. By \eqref{exotic strichartz}, for $v = u - \tilde{u}$,
\begin{align}\label{x-y calc}
\|v\|_{X(I)}
\lesssim \|e^{itH}v(0)\|_{X(I)} + \|f(\tilde{u} + v) - f(\tilde{u})\|_{Y(I)} + \|\int_{0}^{t}e^{i(t - s)H}e(s) ds \|_{X(I)}.
\end{align}

We first estimate the first term on the right-hand side of \eqref{x-y calc}. By the Gagliardo--Nirenberg and H\"older inequalities,
\begin{align*}
\|e^{itH}v(0)\|_{X(I)}
&\lesssim \left\| \|\Delta e^{itH}v(0)\|_{L^{r}_{x}}^{\theta} \|e^{itH}v(0)\|_{L^{\frac{2(n+4)}{n-4}}_{x}}^{1 - \theta} \right\|_{L^{\frac{n^{2} - 16}{8}}_{t}} \\
&\lesssim \|\Delta e^{itH}v(0) \|_{L^{b}(I, L^{r})}^{\theta} \|e^{itH}v(0)\|_{Z(I)}^{1 - \theta},
\end{align*}
where
\[
\theta = \dfrac{n^{2} -8n +14}{(n - 4)^{2}},\qquad
b = \dfrac{2(n + 4)(n^{2} - 8n + 14)}{14(n - 4)},\qquad
\dfrac{1}{r} = \dfrac{1}{2} - \dfrac{4}{n}\cdot\dfrac{1}{b}.
\]
Hence, by the Strichartz estimate \eqref{strichartz}, \eqref{close to initial data}, and \eqref{small condition 2},
\begin{equation}\label{stablity:1st}
\begin{aligned}
\|\Delta e^{itH}v(0) \|_{L^{b}(I, L^{r})}^{\theta} \|e^{itH}v(0)\|_{Z(I)}^{1 - \theta}
&\lesssim \|v(0)\|_{\dot{H}^{2}}^{\theta}\|e^{itH}v(0)\|_{W(I)}^{1 - \theta} \\
&\lesssim (E')^{\theta}\varepsilon^{1 - \theta}.
\end{aligned}
\end{equation}

Next we estimate the second term on the right-hand side of \eqref{x-y calc}. For this purpose we prove the following lemma.

\begin{lemma}[Nonlinear estimate]
\begin{equation}\label{nonliner estimate}
\|f_{z}(v)u\|_{Y(I)} \lesssim \|v\|_{W(I)}\|u\|_{X(I)}.
\end{equation}
\end{lemma}

\begin{proof}
Using Appendix \ref{Product rule} and Sobolev's inequality, we obtain
\begin{equation}\label{nonliner:1}
\begin{aligned}
\||\nabla|^{\frac{8n}{n^{2}-16}}(f_{z}(v)u))\|_{L^{\frac{2(n+4)}{n+8}}}
&\lesssim \||\nabla|^{\frac{8n}{n^{2}-16}}f_{z}(v)\|_{L^{\frac{n^{2} - 16}{4(n-2)}}}\|u\|_{L^{\frac{2(n^{2}-16)}{n^{2}-4n-16}}} \\
&\quad + \|f_{z}(v)\|_{L^{\frac{n+4}{4}}}\||\nabla|^{\frac{8n}{n^{2}-16}}u\|_{L^{\frac{2(n+4)}{n}}} \\
&\lesssim \||\nabla|^{\frac{8n}{n^{2}-16}}f_{z}(v)\|_{L^{\frac{n^{2} - 16}{4(n-2)}}}\||\nabla|^{\frac{8n}{n^{2}-16}}u\|_{L^{\frac{2(n+4)}{n}}} \\
&\quad + \|v\|_{L^{\frac{2(n+4)}{n-4}}}^{\frac{8}{n-4}}\||\nabla|^{\frac{8n}{n^{2}-16}}u\|_{L^{\frac{2(n+4)}{n}}}.
\end{aligned}
\end{equation}
Next, using Appendix \ref{chain}, H\"older, and Sobolev, we have
\begin{equation}\label{nonliner:2}
\begin{aligned}
\||\nabla|^{\frac{8n}{n^{2}-16}}f_{z}(v)\|_{L^{\frac{n^{2} - 16}{4(n-2)}}}
&\lesssim \||v|^{\frac{16}{(n+2)(n-4)}}\|_{L^{\frac{(n+2)(n-4)}{8}}}\||\nabla|^{\frac{n+2}{n+4}}v\|_{L^{\frac{2n(n+4)}{n^{2}-2n+4}}}^{\frac{8n}{(n+2)(n-4)}}\\
&\lesssim \|v\|_{L^{\frac{2(n+4)}{n-4}}}^{\frac{16}{(n+2)(n-4)}}\|\nabla v\|_{L^{\frac{2n(n+4)}{n^{2}-2n+8}}}^{\frac{8n}{(n+2)(n-4)}} \\
&\lesssim \|\nabla v\|_{L^{\frac{2n(n+4)}{n^{2}-2n+8}}}^{\frac{8}{n-4}}.
\end{aligned}
\end{equation}
Combining \eqref{nonliner:1} and \eqref{nonliner:2}, we obtain
\[
\||\nabla|^{\frac{8n}{n^{2}-16}}(f_{z}(v)u))\|_{L^{\frac{2(n+4)}{n+8}}}
\lesssim  \|\nabla v\|_{L^{\frac{2n(n+4)}{n^{2}-2n+8}}}^{\frac{8}{n-4}}\||\nabla|^{\frac{8n}{n^{2}-16}}u\|_{L^{\frac{2(n+4)}{n}}}.
\]
Applying H\"older in time gives \eqref{nonliner estimate}.
\end{proof}

Continuing the proof,
\[
 f(\tilde{u} + v) - f(\tilde{u}) = \int_{0}^{1}[f_{z}(\tilde{u} + \theta v)v + f_{\bar{z}}(\tilde{u} + \theta v)\bar{v}]d\theta.
\]
By Minkowski's inequality and \eqref{nonliner estimate},
\begin{equation}\label{estimate f with y norm}
\|f(\tilde{u} + v) - f(\tilde{u})\|_{Y(I)} \lesssim (\|\tilde{u}\|^{\frac{8}{n - 4}}_{W(I)} + \|u\|^{\frac{8}{n - 4}}_{W(I)} )\|v\|_{X(I)}.
\end{equation}
On the other hand, by the Strichartz estimate \eqref{strichartz} together with \eqref{small condition 1}, \eqref{small condition 2}, and \eqref{small condition 3},
\begin{align*}
\|e^{itH}u_{0}\|_{W(I)}
&\leq \|e^{itH}v(0)\|_{W(I)} + \|e^{itH}\tilde{u}(0) \|_{W(I)} \\
&\lesssim \varepsilon + \|\tilde{u}\|_{W(I)} + \|\tilde{u}\|_{W(I)}^{\frac{n + 4}{n - 4}} + \|e\|_{N(I)} \\
&\lesssim \varepsilon + \varepsilon_{0} + \varepsilon_{0}^{\frac{n + 4}{n - 4}} + \varepsilon \lesssim \varepsilon_{0}.
\end{align*}
Hence Proposition \ref{local existence} yields
\begin{equation}\label{estimate u with w norm}
\|u\|_{W(I)} \lesssim \varepsilon_{0}.
\end{equation}
Substituting \eqref{estimate u with w norm} and \eqref{small condition 1} into \eqref{estimate f with y norm}, we obtain
\begin{equation}\label{stablity:2nd}
\|f(\tilde{u} + v) - f(\tilde{u})\|_{Y(I)} \lesssim \varepsilon_{0}^{\frac{8}{n - 4}}\|v\|_{X(I)}.
\end{equation}

Finally, we estimate the third term on the right-hand side of \eqref{x-y calc}. By \eqref{strichartz} and Sobolev,
\begin{equation}\label{stablity:3rd}
\begin{aligned}
\|\int_{0}^{t}e^{i(t - s)H}e(s) ds \|_{X(I)}
&\lesssim \|\Delta \int_{0}^{t}e^{i(t - s)H}e(s) ds \|_{L^{\frac{n^{2} - 16}{8}}(I, L^{\frac{n^{3} - 16n -64}{2n(n^{2}-16)}})} \\
&\lesssim \|e\|_{N(I)} \lesssim \varepsilon.
\end{aligned}
\end{equation}
Substituting \eqref{stablity:1st}, \eqref{stablity:2nd}, and \eqref{stablity:3rd} into \eqref{x-y calc}, we obtain
\[
\|v\|_{X(I)} \lesssim (E')^{\theta}\varepsilon^{1 - \theta} + \varepsilon_{0}^{\frac{8}{n - 4}}\|v\|_{X(I)} + \varepsilon.
\]
Choosing $\varepsilon_{0}$ sufficiently small, we obtain
\begin{equation}\label{estimate v with x}
\|v\|_{X(I)} \lesssim (E')^{\theta}\varepsilon^{1 - \theta}.
\end{equation}

Next, we improve this bound to the $W$-norm.
By \eqref{strichartz},
\begin{align*}
\|v\|_{W(I)}
&\lesssim \|e^{itH}v(0)\|_{W(I)} + \|f(\tilde{u} + v) - f(\tilde{u})\|_{N(I)} + \|e\|_{N(I)} \\
&\lesssim \varepsilon + \|f(\tilde{u} + v) - f(\tilde{u})\|_{N(I)}.
\end{align*}
Thus we estimate $\|f(\tilde{u} + v) - f(\tilde{u})\|_{N(I)}$.
By \eqref{nonlinear with delivative} and H\"older,
\begin{equation}\label{estimate f with N}
\begin{aligned}
\|f(\tilde{u} + v) - f(\tilde{u})\|_{N(I)}
&\lesssim \|\nabla \tilde{u}\|_{L^{a}(I, L^{b})}\|v\|_{L^{\frac{n^{2} - 16}{8}}(I, L^{\frac{2(n^{2}-16)}{n^{2} - 4n - 16}})}^{\frac{8}{n - 4}}
+ \|u\|_{W(I)}^{\frac{8}{n-4}}\|v\|_{W(I)},
\end{aligned}
\end{equation}
where
\[
\dfrac{1}{a} = \dfrac{1}{2} - \dfrac{64}{(n-4)^{2}(n+4)},\qquad
\dfrac{1}{b} = \dfrac{n+2}{2n} - \dfrac{4(n^{2}-4n-16)}{(n-4)^{2}(n+4)}.
\]
Let $b^{*}$ be defined by $\dfrac{1}{b^{*}} = \dfrac{1}{b} - \dfrac{1}{n}$. Then Sobolev implies
\[
\|\nabla \tilde{u}\|_{L^{a}(I, L^{b})} \lesssim \|\Delta \tilde{u}\|_{L^{a}(I, L^{b^{*}})}.
\]
Since $(a, b^{*})$ is B-admissible, \eqref{strichartz} gives
\begin{equation}\label{calc st}
\begin{aligned}
\|\Delta \tilde{u}\|_{L^{a}(I, L^{b^{*}})}
&\lesssim \|\tilde{u}(0)\|_{\dot{H}^{2}} + \|f(\tilde{u})\|_{N(I)} + \|e\|_{N(I)} \\
&\lesssim E + \|\tilde{u}\|^{\frac{n+4}{n-4}}_{W(I)} + \varepsilon \\
&\lesssim E + \varepsilon_{0}^{\frac{n+4}{n-4}} + \varepsilon \lesssim E.
\end{aligned}
\end{equation}
Also, by Sobolev,
\[
\|v\|_{L^{\frac{n^{2} - 16}{8}}(I, L^{\frac{2(n^{2}-16)}{n^{2} - 4n - 16}})} \lesssim \|v\|_{X(I)}.
\]
Thus, from \eqref{estimate v with x}, \eqref{estimate f with N}, and \eqref{calc st}, we obtain
\begin{equation}\label{stability:ineq:1}
\|f(\tilde{u} + v) - f(\tilde{u})\|_{N(I)} \lesssim E(E')^{\frac{8}{n-4}\theta}\varepsilon^{\frac{8}{n-4}(1 - \theta)} + \varepsilon_{0}^{\frac{8}{n-4}}\|v\|_{W(I)}.
\end{equation}
Hence
\[
 \|v\|_{W(I)} \lesssim \varepsilon + E(E')^{\frac{8}{n-4}\theta}\varepsilon^{\frac{8}{n-4}(1 - \theta)} + \varepsilon_{0}^{\frac{8}{n-4}}\|v\|_{W(I)}.
\]
If $\varepsilon$ is sufficiently small, then
\[
E(E')^{\frac{8}{n-4}\theta}\varepsilon^{\frac{8}{n-4}(1 - \theta)} \lesssim \varepsilon^{\frac{15}{(n-4)^{3}}},
\]
and by Appendix \ref{bootstrap}, \eqref{short stability conclusion 1} follows.
Then \eqref{short stability conclusion 4} follows from \eqref{stability:ineq:1} and \eqref{short stability conclusion 1}. Finally, \eqref{short stability conclusion 2} follows from \eqref{strichartz}, and \eqref{short stability conclusion 3} follows from \eqref{calc st} and the triangle inequality.
\end{proof}

\begin{proposition}[Long time perturbation]\label{prop:long time}
Assume that $V$ satisfies Assumptions \ref{as:wave} and \ref{as:equ}. Let $n \geq 5$ and $p = 2^{\sharp} - 1$.
Let $I \subset \mathbb{R}$ be a compact time interval containing $0$, and let $\tilde{u}$ be an approximate solution to \eqref{eq:NL4S} on $I \times \mathbb{R}^{n}$, namely,
\[
i\partial_{t}\tilde{u} + \Delta^{2}\tilde{u} + V\tilde{u} + \lambda|\tilde{u}|^{\frac{8}{n-4}}\tilde{u} = e
\]
for some $e \in N(I)$.
Assume that for some $M,E >0$,
\begin{align}
\|\tilde{u}\|_{L^{\infty}(I, \dot{H^{2}})} &\leq E, \\
\|\tilde{u}\|_{Z(I)} &\leq M.
\end{align}
Let $u_{0} \in H^{2}$ satisfy
\[
\|u_{0} - \tilde{u}(0) \|_{\dot{H}^{2}} \leq E'
\]
for some $E' >0$.
Assume also that
\begin{align}
\|e^{itH}\left(u_{0} - \tilde{u}(0)\right)\|_{W(I)} \leq \varepsilon, \\
\|e\|_{N(I)} \leq \varepsilon
\end{align}
for some $0 \leq \varepsilon \leq \varepsilon_{1}$, where $\varepsilon_{1} = \varepsilon_{1}(E, E', M) > 0$ is sufficiently small.
Then there exists a solution $u\in C(I, H^{2})$ to \eqref{eq:NL4S} with initial data $u_{0}$ such that
\begin{align}
\|u - \tilde{u}\|_{W(I)} &\lesssim \varepsilon + \varepsilon^{\frac{15}{(n-4)^{3}}}, \label{long stability conclusion 1}\\
\|u - \tilde{u}\|_{L^{q}(I, \dot{H}^{2, r})} &\lesssim E' + \varepsilon  + \varepsilon^{\frac{15}{(n-4)^{3}}},\label{long stability conclusion 2}\\
\|u\|_{L^{q}(I, \dot{H}^{2, r})} &\lesssim E + E', \label{long stability conclusion 3}\\
\|(i\partial_{t} + \Delta^{2} + V)(u - \tilde{u}) + e\|_{N(I)} &\lesssim \varepsilon + \varepsilon^{\frac{15}{(n-4)^{3}}}. \label{long stability conclusion 4}
\end{align}
\end{proposition}

\begin{proof}
Let $N_{0} \sim \left( 1 + \frac{M}{\varepsilon_{0}} \right)^{\frac{2(n+4)}{n-4}}$, and decompose $I$ into $N_{0}$ subintervals so that
\[
\|\tilde{u}\|_{Z(J_{k})} \leq \varepsilon_{0}.
\]
Then, by \eqref{strichartz},
\begin{align*}
\|\tilde{u}\|_{W(J_{k})}
&\lesssim \|\tilde{u}(0)\|_{\dot{H}^{2}} + \|\tilde{u}\|_{Z(J_{k})}^{\frac{8}{n-4}}\|\tilde{u}\|_{W(J_{k})} + \|e\|_{N(J_{k})} \\
&\lesssim E + \varepsilon_{0}^{\frac{8}{n-4}}\|\tilde{u}\|_{W(J_{k})} + \varepsilon.
\end{align*}
Hence, if $\varepsilon_{0}$ is sufficiently small,
\[
\|\tilde{u}\|_{W(J_{k})} \lesssim E.
\]
Summing over all $J_{k}$, we obtain
\[
\|\tilde{u}\|_{W(I)} \leq C(E, M, \varepsilon_{0}).
\]

Now divide $I$ into $N_{1} \sim C(E, M, \varepsilon_{0})$ subintervals $I_{j} = [t_{j}, t_{j + 1}]$ such that
\[
\|\tilde{u}\|_{W(I_{j})} \leq \varepsilon_{0}.
\]
If $\varepsilon_{1}$ is chosen sufficiently small depending on $N_{1}, E, E'$, then Proposition \ref{Short time perturbation} implies that whenever
\begin{equation}\label{long:1}
\begin{aligned}
\|u(t_{j}) - \tilde{u}(t_{j}) \|_{\dot{H}^{2}} &\leq E', \\
\|e^{i(t - t_{j})H}\left(u(t_{j}) - \tilde{u}(t_{j})\right)\|_{W(I_{j})} &\leq \varepsilon,
\end{aligned}
\end{equation}
we have
\begin{align*}
\|u - \tilde{u}\|_{W(I_{j})} &\leq C(j)\left(\varepsilon + \varepsilon^{\frac{15}{(n-4)^{3}}}\right),  \\
\|u - \tilde{u}\|_{L^{q}(I_{j}, \dot{H}^{2, r})} &\leq C(j)\left(E' + \varepsilon  + \varepsilon^{\frac{15}{(n-4)^{3}}}\right), \\
\|u\|_{L^{q}(I_{j}, \dot{H}^{2, r})} &\leq C(j)\left(E + E'\right), \\
\|(i\partial_{t} + \Delta^{2} + V)(u - \tilde{u}) + e\|_{N(I_{j})} &\leq C(j)\left(\varepsilon + \varepsilon^{\frac{15}{(n-4)^{3}}}\right).
\end{align*}
We now prove \eqref{long:1} by induction starting from $t_{0}=0$.
Indeed,
\begin{align*}
\|u(t_{j+1}) - \tilde{u}(t_{j+1}) \|_{\dot{H}^{2}}
&\lesssim \|u(t_{0}) - \tilde{u}(t_{0}) \|_{\dot{H}^{2}} + \|e\|_{N([t_{0}, t_{j+1}])} + \|f(u) - f(\tilde{u})\|_{N([t_{0}, t_{j+1}])} \\
&\lesssim E' + \varepsilon + \sum_{k = 0}^{j}C(k)\left(\varepsilon + \varepsilon^{\frac{15}{(n-4)^{3}}}\right) \lesssim E'
\end{align*}
and similarly
\begin{align*}
\|e^{i(t - t_{j+1})H}\left(u(t_{j+1}) - \tilde{u}(t_{j+1})\right)\|_{W(I)}
&\lesssim \|e^{i(t - t_{0})H}\left(u(t_{0}) - \tilde{u}(t_{0})\right)\|_{W(I)} + \|e\|_{N(I)} + \|f(u) - f(\tilde{u})\|_{N(I)} \\
&\lesssim \varepsilon + \sum_{k = 0}^{j}C(k)\left(\varepsilon + \varepsilon^{\frac{15}{(n-4)^{3}}}\right).
\end{align*}
This proves the conclusion.
\end{proof}

\section{Almost conservation laws}

In this section we establish the almost conservation of localized mass and a localized Morawetz-type estimate.
We begin with the almost conservation of localized mass.

Let $\chi \in C^{\infty}_{0}(\mathbb{R}^{n})$ satisfy $\chi(r)=1$ for $r \leq 1$, $\chi(r)=0$ for $r \geq 2$, and $0 \leq \chi \leq 1$.
For $u \in L^{2}$, define the mass of $u$ on the ball $B_{x_{0}}(R)$ by
\begin{equation}\label{localize mass}
M(u, B_{x_{0}}(R)) = \int_{\mathbb{R}^{n}}|u(x)|^{2}\chi_{R}^{4}(x - x_{0})dx,
\end{equation}
where $g_{R}(x) = g(x /R)$.
By H\"older's inequality and Sobolev's inequality,
\begin{equation}\label{estimate for local mass 1}
M(u, B_{x_{0}}(R)) \lesssim \|\Delta u\|^{2}_{L^{2}}R^{4}.
\end{equation}

We now show that, for a solution $u \in C(I, H^{2})$ to \eqref{eq:NL4S}, the localized mass varies slowly in time when the radius $R$ is sufficiently large.

\begin{lemma}
Let $n \geq 5$ and $p \in (1, 2^{\sharp} - 1]$. Let $\lambda \in \mathbb{R}$ and let $u \in C(I, H^{2})$ be a solution to \eqref{eq:NL4S}.
Then, for any $t \in I$,
\begin{equation}\label{estimate for local mass 2}
|\partial_{t}M(u(t), B_{x_{0}}(R))| \leq C\dfrac{\mathcal{E}(u)^{\frac{3}{4}}}{R}M(u(t), B_{x_{0}}(R))^{\frac{1}{4}}
\end{equation}
holds, where $C > 0$ is independent of $u$ and $I$.
\end{lemma}

\begin{proof}
By translation invariance, we may assume $x_{0}=0$.
Using \eqref{eq:NL4S} and integration by parts,
\begin{equation}\label{loacal math derivative}
\begin{aligned}
\dfrac{d}{dt}M(u(t), B_{0}(R))
&= 2\text{Re}\int_{\mathbb{R}^{n}} i((\Delta^{2}u) \bar{u} + V|u|^{2} + \lambda|u|^{p})\chi_{R}^{4}(x)dx \\
&= 2\text{Re}\int_{\mathbb{R}^{n}} i(\Delta^{2}u) \bar{u} \chi_{R}^{4}(x)dx \\
&= \dfrac{16}{R} \text{Re}\int_{\mathbb{R}^{n}} i\Delta u \nabla \bar{u}(\nabla \chi)_{R} \chi_{R}^{3}dx \\
&\quad + \dfrac{8}{R^{2}} \text{Re} \int_{\mathbb{R}^{n}} i\Delta u\left( \chi_{R}^{3} (\Delta \chi)_{R} + 3\chi_{R}^{2}(\nabla \chi)_{R}^{2}\right)\bar{u} dx.
\end{aligned}
\end{equation}
We estimate the two terms in the last line separately.
By the Cauchy--Schwarz inequality,
\begin{align*}
\left|\int_{\mathbb{R}^{n}} \Delta u \nabla \bar{u}(\nabla \chi)_{R} \chi_{R}^{3}dx \right|
\leq \|(\nabla \chi)_{R}\|_{L^{\infty}}\left( \int_{\mathbb{R}^{n}}|\Delta u|^{2} dx\right)^{\frac{1}{2}}\left(\int_{\mathbb{R}^{n}} |\nabla u|^{2}\chi_{R}^{6} dx\right)^{\frac{1}{2}}.
\end{align*}
Also, integrating by parts,
\[
 \int_{\mathbb{R}^{n}} |\nabla u|^{2}\chi_{R}^{6} dx
 = -\int_{\mathbb{R}^{n}} \bar{u}\left( (\Delta u) \chi_{R}^{6} + \nabla u\dfrac{6}{R}(\nabla \chi)_{R} \chi_{R}^{5} \right) dx.
\]
By H\"older and Sobolev,
\begin{equation}
\begin{aligned}
\int_{\mathbb{R}^{n}} |\nabla u|^{2}\chi_{R}^{6} dx
&\leq \|\chi_{R}\|_{L^{\infty}}^{4} \left(\int_{\mathbb{R}^{n}} |u|^{2}\chi_{R}^{4} dx\right)^{\frac{1}{2}} \left(\int_{\mathbb{R}^{n}} |\Delta u|^{2}dx \right)^{\frac{1}{2}} \\
&\quad + \dfrac{6}{R} \left(\int_{\mathbb{R}^{n}} |u|^{2}\chi_{R}^{4} dx\right)^{\frac{1}{2}} \left(\int_{\mathbb{R}^{n}} |\nabla u|^{2}(\nabla \chi_{R})^{2}\chi_{R}^{6} dx\right)^{\frac{1}{2}} \\
&\lesssim \|u\|_{\dot{H}^{2}}M(u(t), B_{0}(R))^{\frac{1}{2}} + \dfrac{1}{R}M(u(t), B_{0}(R))^{\frac{1}{2}} \\
&\quad \times \left(\int_{\mathbb{R}^{n}} |\nabla u|^{\frac{2n}{n-2}}dx\right)^{\frac{n-2}{2n}} \left(\int_{\mathbb{R}^{n}}(\nabla \chi_{R})^{n}\chi_{R}^{3n} dx\right)^{\frac{1}{n}} \\
&\lesssim \|u\|_{\dot{H}^{2}}M(u(t), B_{0}(R))^{\frac{1}{2}} + \dfrac{1}{R}M(u(t), B_{0}(R))^{\frac{1}{2}}\|u\|_{\dot{H}^{2}}R \\
&\lesssim \|u\|_{\dot{H}^{2}}M(u(t), B_{0}(R))^{\frac{1}{2}}.
\end{aligned}
\end{equation}

Next, for the second term in \eqref{loacal math derivative}, using \eqref{estimate for local mass 1}, H\"older, and Sobolev,
\begin{align*}
\left|\int_{\mathbb{R}^{n}} \Delta u\chi_{R}^{2}\bar{u}\left( \chi_{R} (\Delta \chi)_{R} + 3(\nabla \chi)_{R}^{2}\right) dx\right|
&\lesssim \left| \int_{\mathbb{R}^{n}} \Delta u \chi_{R}^{2} \bar{u} dx \right| \\
&\lesssim \left( \int_{\mathbb{R}^{n}} |\Delta u|^{2}dx \right)^{\frac{1}{2}} \left(\int_{\mathbb{R}^{n}} |u|^{2}\chi_{R}^{4}dx \right)^{\frac{1}{2}} \\
&\lesssim \|u\|_{\dot{H}^{2}}M(u(t), B_{0}(R))^{\frac{1}{2}} \\
&\lesssim R\|u\|_{\dot{H}^{2}}^{\frac{3}{2}}M(u(t), B_{0}(R))^{\frac{1}{4}}.
\end{align*}
Since $\|u\|_{\dot{H}^{2}} \leq \mathcal{E}(u)^{1/2}$, \eqref{estimate for local mass 2} follows.
\end{proof}

We next consider a localized Morawetz-type estimate.

\begin{proposition}
Let $n \geq 5$ and $p = 2^{\sharp} - 1$, and assume that $V$ satisfies Assumption \ref{as:mora}.
Then, for any $T > 0$, $u \in C([0, T], H^{2})$, $K > 0$, and $I \subset [0, T]$,
\begin{equation}\label{Morawetz}
\int_{I}\int_{|x| \leq K|I|^{\frac{1}{4}}} \dfrac{|u(x)|^{2^{\sharp}}}{|x|}dx \leq C(K^{3} + K^{-1})\left( \sup_{I} \hat{\mathcal{E}}(u) \right)|I|^{\frac{3}{4}},
\end{equation}
where $\hat{\mathcal{E}}(u) = \mathcal{E}(u) + \mathcal{E}(u)^{2^{\sharp}/2}$.
\end{proposition}

\begin{proof}
Let $u_{0} \in C_{0}^{\infty}(\mathbb{R}^{n})$, $h \in C_{0}^{\infty}(\mathbb{R}^{n + 1})$, and let $v$ be the solution of \eqref{eq:duhamel}.
For two differentiable functions $f,g$, define
\begin{equation}
\{f, g\}_{p} = \text{Re}(f\nabla \bar{g} - g \nabla \bar{f}).
\end{equation}
For a compactly supported real-valued function $a$, define
\begin{equation}\label{def local mass}
M_{a}^{0}(t) = 2 \int_{\mathbb{R}^{n}} \partial_{j}a(x)\text{Im}(\bar{v}(t,x)\partial_{j}v(t, x))dx.
\end{equation}
We first compute the time derivative of $M_{a}^{0}$.
Let
\begin{align*}
T_{j} &= \text{Im}(\bar{v}\partial_{j}v), \\
T_{j}^{\alpha} &= - \{h, v \}^{j}_{p}, \\
T_{j}^{\beta} &= \dfrac{1}{2}\left(\bar{v}\partial_{j}(\Delta^{2}v) + v\partial_{j}(\Delta^{2}\bar{v}) -\partial_{j}v\Delta^{2}\bar{v} - \partial_{j}\bar{v}\Delta^{2}v\right), \\
T_{j}^{\gamma} &= \dfrac{1}{2}\left(\bar{v}\partial_{j}(Vv) + v\partial_{j}(V\bar{v}) -\partial_{j}vV\bar{v} - \partial_{j}\bar{v}Vv\right).
\end{align*}
Then
\[
 \partial_{t}T_{j} = T_{j}^{\alpha} + T_{j}^{\beta} + T_{j}^{\gamma}.
\]
If we set
\[
T_{j,k} = 2\Delta\text{Re}(\partial_{j}v\partial_{k}\bar{v}) - \dfrac{1}{2}\delta_{j,k}\Delta^{2}(|v|^{2}) + \delta_{j,k}\Delta(|\nabla v|^{2}) - \sum_{i}^{n}4\text{Re}(\partial_{j,i}\bar{v}\partial_{k,i}v),
\]
then a direct computation shows that
\[
T_{j}^{\beta} = \sum_{k}^{n}\partial_{k}T_{j,k}.
\]
Also,
\begin{align*}
T_{j}^{\gamma}
&= -\text{Re}(Vv \partial_{j} \bar{v} - v\partial_{j}(V\bar{v})) \\
&= |v|^{2}\partial_{j} V.
\end{align*}
Hence, integrating by parts,
\begin{equation}\label{partial mass}
\begin{aligned}
\partial_{t} M_{a}^{0} (t)
&= 2\int_{\mathbb{R}^{n}} \partial_{j}a(x)T_{j}^{\alpha} + 2\int_{\mathbb{R}^{n}} \partial_{j}a(x)T_{j}^{\beta} + 2\int_{\mathbb{R}^{n}} \partial_{j}a(x)T_{j}^{\gamma} \\
&= 2\int_{\mathbb{R}^{n}}\Bigl( - 2\text{Re}(\partial_{j}v\partial_{k}\bar{v}) \partial_{j,k}\Delta a + \dfrac{1}{2}(\Delta^{3}a)|v|^{2}
+ 4\partial_{j,k}a\text{Re}(\partial_{j,i}\bar{v}\partial_{k,i}v) \\
&\qquad -\Delta^{2}a|\nabla v|^{2}
- \partial_{j}a \{h, v \}^{j}_{p} + \partial_{j}a |v|^{2}\partial_{j} V\Bigr)dx.
\end{aligned}
\end{equation}
By density, \eqref{partial mass} also holds for $h\in N(I)$ and $v\in C(I,H^{2})$.

Now let $u\in C(I, H^{2}) \cap M(I)$ be a solution to \eqref{eq:NL4S}. Then $u$ is a solution to \eqref{eq:duhamel} with
\[
h = \lambda |u|^{2^{\sharp}-2}u,\qquad h\in N(I).
\]
Hence \eqref{partial mass} holds for this $u,h$. Moreover,
\begin{equation}\label{nonlinear estimate in mass}
\int_{\mathbb{R}^{n}}\partial_{j}a\{h, u\}_{p}^{j}dx = \dfrac{4\lambda}{n}\int_{\mathbb{R}^{n}}(\Delta a)|u|^{2^{\sharp}} dx.
\end{equation}

Now let $a(x) = \langle x \rangle_{\delta}\chi_{R}(x)$, where
\[
\langle x \rangle_{\delta} = (\delta^{2} + |x|^{2})^{\frac{1}{2}}.
\]
If $\alpha \in \mathbb{N}^{n}$ is a multi-index and $R \geq \delta$, $R \leq |x| \leq 2R$, then
\begin{equation}
|D^{\alpha}a(x)| \leq CR^{1 - |\alpha|}.
\end{equation}
Integrating \eqref{partial mass} over $I$ and using \eqref{def local mass} and \eqref{nonlinear estimate in mass}, we obtain
\begin{equation}
\begin{aligned}
&2\int_{I}\int_{|x| \leq R}\left( \dfrac{4\sum_{i}(\nabla|\partial_{j}u|^{2} - |\partial_{r}\partial_{i}u|^{2})}{\langle x \rangle_{\delta}} + \dfrac{2(n-1)(|\nabla u|^{2}-3|\partial_{r}u|^{2})}{\langle x \rangle_{\delta}^{3}}\right) dx \\
&+ \int_{I}\int_{|x| \leq R}\left( \dfrac{(n-1)(n-3)|\nabla u|^{2}}{\langle x \rangle_{\delta}} + \dfrac{8\lambda (n - 1)|u|^{2^{\sharp}}}{n\langle x \rangle_{\delta}}\right) dx + O(\delta)\\
&- \int_{I}\int_{|x| \leq R} \dfrac{x \cdot \nabla V}{|x|}|u|^{2} dx \\
&\leq C\int_{I}\int_{R \leq |x| \leq 2R} \left(R^{-3}|\nabla u|^{2} + R^{-5}|u|^{2} + R^{-1}|\nabla^{2}u|^{2} + R^{-1}|u|^{2^{\sharp}} + |\nabla V||u|^{2}\right)dx \\
&+ C\int_{|x| \leq 2R}\left[ |u\nabla u| \right]_{t_{1}}^{t_{2}}dx.
\end{aligned}
\end{equation}
Using the assumption $|\nabla V| \lesssim R^{-5}$ for $R \leq |x| \leq 2R$ and letting $\delta \to 0$, we obtain
\begin{equation}
\begin{aligned}
&2\int_{I}\int_{|x| \leq R}\left( \dfrac{4\sum_{i}(\nabla|\partial_{j}u|^{2} - |\partial_{r}\partial_{i}u|^{2})}{|x|} + \dfrac{2(n-1)(|\nabla u|^{2}-3|\partial_{r}u|^{2})}{|x|^{3}}\right) dx \\
&+ \int_{I}\int_{|x| \leq R}\left( \dfrac{(n-1)(n-3)|\nabla u|^{2}}{|x|} + \dfrac{8\lambda (n - 1)|u|^{2^{\sharp}}}{n|x|}\right) dx\\
&- \int_{I}\int_{|x| \leq R} \dfrac{x \cdot \nabla V}{|x|}|u|^{2} dx \\
&\leq C\int_{I}\int_{R \leq |x| \leq 2R} \left(R^{-3}|\nabla u|^{2} + R^{-5}|u|^{2} + R^{-1}|\nabla^{2}u|^{2} + R^{-1}|u|^{2^{\sharp}} + R^{-5}|u|^{2}\right)dx \\
&+ C\int_{|x| \leq 2R}\left[ |u\nabla u| \right]_{t_{1}}^{t_{2}}dx \\
&\leq C|I|R\sup_{I}\left(\mathcal{E}(u) + \mathcal{E}(u)^{2^{\sharp}/2}\right) + CR^{3}\sup_{I}\mathcal{E}(u).
\end{aligned}
\end{equation}
Here $C$ is independent of $I,u,R$.
Moreover, it is shown in \cite{Leva} that for $u \in H^{2}$,
\begin{equation}\label{eq:leva}
\sum_{i}\left( |\nabla \partial_{i}u|^{2} - |\partial_{r}\partial_{i}u|^{2}\right) \geq \dfrac{n-1}{|x|^{2}}|\partial_{r}u|^{2}.
\end{equation}
Hence \eqref{eq:leva} implies
\begin{align*}
\int_{I}\int_{|x| \leq R}\left( \dfrac{4\sum_{i}(\nabla|\partial_{j}u|^{2} - |\partial_{r}\partial_{i}u|^{2})}{|x|} + \dfrac{2(n-1)(|\nabla u|^{2}-3|\partial_{r}u|^{2})}{|x|^{3}}\right) dx \geq 0.
\end{align*}
Since $x \cdot \nabla V \leq 0$, taking $R = K|I|^{\frac{1}{4}}$ yields \eqref{Morawetz}.
\end{proof}

\section{Global existence}

In this section we prove Theorem \ref{th:global}. Let $H^{2}_{rad}$ denote the set of radial functions in $H^{2}$.

\begin{proposition}\label{prepare main thm}
Let $V \in L^{\infty} \cap L^{\frac{n}{4}}$ be a radial real-valued potential satisfying Assumptions \ref{as:wave}, \ref{as:mora}, and \ref{as:equ}. Let $n \geq 5$, $p = 2^{\sharp} - 1$, and assume $\lambda > 0$.
Let $u \in C([t_{-}, t_{+}], H^{2}_{rad})$ be a radial solution to \eqref{eq:NL4S} satisfying $\|u\|_{W([t_{-}, t_{+}])} < \infty$.
Then there exists a constant $K > 0$, depending only on $n$, $\lambda$, and $\mathcal{E} = \sup_{t}\mathcal{E}(u)$, such that
\begin{equation}\label{bound of z norm}
\|u\|_{Z([t_{-}, t_{+}])} \leq K.
\end{equation}
\end{proposition}

We first show that Proposition \ref{prepare main thm} implies Theorem \ref{th:global}.

\begin{proof}[Proof of Theorem \ref{th:global}]
Let $u_{0} \in H^{2}$ be radial. By the Strichartz estimate, there exists $T > 0$ such that
\[
\|e^{itH}u_{0}\|_{W(I)} < \delta,\qquad I = [0, T].
\]
Then Proposition \ref{local existence} implies that there exists a solution $u \in C(I, H^{2})$ to \eqref{eq:NL4S} with $u(0)=u_{0}$.
By Proposition \ref{criterion}, this solution can be extended to its maximal interval $[0, T^{*})$, and for every compact interval $I' \subset [0, T^{*})$, one has $u \in M(I')$.

If $T^{*} < \infty$, then the $Z([0,T^{*}))$-norm of $u$ must blow up.
By uniqueness, $u$ remains radial.
Assume for contradiction that $T^{*} < \infty$.
Let $I' \subset [0, T^{*})$ be compact. Then Proposition \ref{prepare main thm} implies that the $Z(I')$-norm of $u$ is bounded by a constant independent of $I'$.
This contradicts the blow-up of the $Z$-norm.
Hence $T^{*} = +\infty$.
\end{proof}

We now prove Proposition \ref{prepare main thm}. The proof is divided into several steps.

\begin{step}\label{step1}
Let $u \in C([t_{1}, t_{2}], H^{2}_{rad})$ be a radial solution to \eqref{duhamel2} on $I = [t_{1}, t_{2}]$.
There exists a sufficiently small constant $\eta_{0} > 0$, depending only on $n$ and $\lambda$, such that if
\begin{equation}\label{small condition in main thm}
\dfrac{1}{4}\eta < \|u\|_{W(I)} < \eta
\end{equation}
holds for some $0 < \eta < \eta_{0}$, then
\begin{equation}\label{small w-norm}
\|u_{k}\|_{W(I)} \geq \dfrac{1}{8} \eta
\end{equation}
holds, where $u_{k} = e^{itH}u(t_{k})$, $k = 1, 2$.
\end{step}

\begin{proof}[Proof of Step \ref{step1}]
By Duhamel's formula and the Strichartz estimate,
\begin{align*}
\|u_{k}\|_{W(I)}
&\geq \|u\|_{W(I)} - C\||u|^{\frac{8}{n-4}}u\|_{N(I)}
\geq \dfrac{1}{4} \eta - C \eta^{\frac{n+4}{n-4}}.
\end{align*}
Hence, if $\eta \leq \eta_{0}$ is sufficiently small, \eqref{small w-norm} follows.
\end{proof}

If $\mathcal{E}(u)$ is small, then global existence follows from Corollary \ref{global solution for small energy}. Thus we may assume that $\mathcal{E}(u)$ is bounded below away from $0$; namely, there exists $\varepsilon_{0}>0$ such that
\[
\forall t \in [t_{1}, t_{2}],\qquad \mathcal{E}(u(t)) \geq \varepsilon_{0}.
\]

We partition $[t_{-}, t_{+}]$ into $N$ pairwise disjoint intervals $(I_{j})_{1\leq j \leq N}$ such that on each interval \eqref{small condition in main thm} holds.
Also define
\[
u_{\pm}(t) = e^{i(t - t_{\pm})H}u(t_{\pm}).
\]
We call an interval $I_{j}$ \emph{exceptional} if one of the following holds:
\begin{equation}\label{exceptional}
\begin{aligned}
\|u_{+}\|_{W(I_{j})} &> \eta^{K_{2}},\ \ \text{or} \\
\|u_{-}\|_{W(I_{j})} &> \eta^{K_{2}},
\end{aligned}
\end{equation}
where $K_{2} = 24n^{2}$.
Intervals that are not exceptional are called \emph{unexceptional}.
By the Strichartz estimate and Sobolev's inequality, the number $N_{e}$ of exceptional intervals satisfies
\begin{equation}\label{bound of exceptional}
N_{e} \leq C(\|u_{0}\|_{\dot{H}^{2}}\eta^{-K_{2}})^{\frac{2(n+4)}{n-4}} + 1.
\end{equation}
If all intervals are exceptional, then \eqref{bound of z norm} follows immediately from \eqref{small condition in main thm} and \eqref{bound of exceptional}.
Hence we may assume from now on that there exists at least one unexceptional interval.
By Step \ref{step1}, the intervals $I_{1}$ and $I_{N}$ are always exceptional.

\begin{step}\label{step2}
Let $u \in C([t_{-}, t_{+}], H^{2}_{rad})$ be a radial solution to \eqref{duhamel2}, and let $I = [t_{0}, t_{1}]$ be an unexceptional interval for $u$.
Then there exists $x_{0} \in \mathbb{R}^{n}$ such that, for every $t \in I$,
\begin{equation}\label{conclusion of step9-2}
M\left(u(t), B_{x_{0}}(2\eta^{-K_{1}}|I|^{\frac{1}{4}}) \right) \geq C\eta^{K_{1}}\mathcal{E}^{-\frac{n+2}{2}}|I|,
\end{equation}
where $K_{1} = n^{2} + 6n + 4$, and $C>0$ is independent of $I$, $x_{0}$, and $u$.
Moreover, $\eta$ is assumed sufficiently small so that $\eta < \mathcal{E}^{-5n}\eta_{1}$, where $\eta_{1} > 0$ depends only on $n$ and $\lambda$.
\end{step}

\begin{proof}[Proof of Step \ref{step2}]
Let
\[
I^{1} = \left[t_{0}, \frac{t_{0} + t_{1}}{2}\right],\qquad
I^{2} = \left[\frac{t_{0} + t_{1}}{2}, t_{1}\right].
\]
By time reversal symmetry, time translation invariance, and \eqref{small condition in main thm}, we may assume
\begin{equation}\label{half}
\|u\|_{W(I^{2})} \geq \dfrac{1}{8}\eta.
\end{equation}
Moreover, by Duhamel's formula, for any $t \in I^{2}$,
\begin{equation}\label{divided duhamel}
\begin{aligned}
u(t) = u_{-}(t) + i\lambda \int_{t_{-}}^{t_{0}} e^{i(t-s)H}|u(s)|^{\frac{8}{n-4}}u(s)ds 
+ i\lambda \int_{t_{0}}^{t} e^{i(t-s)H}|u(s)|^{\frac{8}{n-4}}u(s)ds.
\end{aligned}
\end{equation}
Since $I$ is unexceptional, the first term on the right-hand side is small in the $W$-norm.
Also, by Sobolev and \eqref{strichartz}, the third term satisfies
\begin{equation}\label{estimate duhamel on half}
\begin{aligned}
\|\int_{t_{0}}^{t} e^{i(t-s)H}|u(s)|^{\frac{8}{n-4}}u(s)ds\|_{W(I^{2})}
&\leq C\|\int_{t_{0}}^{t} e^{i(t-s)H}|u(s)|^{\frac{8}{n-4}}u(s)ds\|_{M(I^{2})} \\
&\leq C\||u(s)|^{\frac{8}{n-4}}u\|_{N(I^{2})} \\
&\leq C\|u\|_{W(I^{2})}^{\frac{n+4}{n-4}} \leq C\eta^{\frac{n+4}{n-4}}.
\end{aligned}
\end{equation}
Define
\begin{equation}
v(t) = \int_{t_{-}}^{t_{0}}e^{i(t-s)H}|u(s)|^{\frac{8}{n-4}}u(s)ds.
\end{equation}
Then, by \eqref{exceptional}, \eqref{half}, and \eqref{estimate duhamel on half}, choosing $\eta$ sufficiently small, we obtain
\begin{equation}\label{es:9-step2-2}
\begin{aligned}
\|v\|_{W(I^{2})}
&\geq \|u\|_{W(I^{2})} - C\eta^{\frac{n+4}{n-4}} - \|u_{-}\|_{W(I^{2})} \\
&\geq \dfrac{1}{8}\eta - \eta^{K_{2}} - C\eta^{\frac{n+4}{n-4}} \geq \dfrac{1}{16}\eta.
\end{aligned}
\end{equation}
Also, $v$ solves the linear equation
\begin{equation}\label{linear equation}
i\partial_{t}v + \Delta^{2}v + Vv = 0.
\end{equation}
Furthermore,
\begin{equation}\label{es:9-step2-0}
\begin{aligned}
\|v\|_{L^{\infty}(I, \dot{H}^{2})}
\leq C\|v(t_{0})\|_{\dot{H}^{2}}
\leq C\lambda^{-1}\left( \|u(t_{0})\|_{\dot{H}^{2}} + \|u(t_{-})\|_{\dot{H}^{2}}\right)
\leq 2C\lambda^{-1}\mathcal{E}^{\frac{1}{2}}.
\end{aligned}
\end{equation}
Using \eqref{strichartz}, we also have
\begin{equation}\label{es:9-step2-1}
\begin{aligned}
&\|u\|_{M(I)} \leq C\|u(t_{0})\|_{ \dot{H}^{2}} + C\|u\|_{W(I)}^{\frac{n+4}{n-4}} \leq C\mathcal{E}^{\frac{1}{2}} + C\eta^{\frac{n+4}{n-4}}, \\
&\|u_{-}\|_{M(I)} \leq C\|u(t_{-})\|_{ \dot{H}^{2}} \leq C\mathcal{E}^{\frac{1}{2}}.
\end{aligned}
\end{equation}
Hence, by \eqref{divided duhamel}, \eqref{estimate duhamel on half}, and \eqref{es:9-step2-1},
\begin{equation}\label{es:9-step2-3}
\begin{aligned}
\|v\|_{M(I^{2})}
&\leq \|u\|_{M(I^{2})} + \|u_{-}\|_{M(I^{2})} + \|\int_{t_{0}}^{t} e^{i(t-s)H}|u(s)|^{\frac{8}{n-4}}u(s)ds\|_{M(I^{2})} \\
&\leq C\mathcal{E}^{\frac{1}{2}} + C\eta^{\frac{n+4}{n-4}} + C\mathcal{E}^{\frac{1}{2}} +  C\eta^{\frac{n+4}{n-4}}
\leq 3C\mathcal{E}^{\frac{1}{2}}.
\end{aligned}
\end{equation}
By the Gagliardo--Nirenberg inequality,
\begin{equation}\label{es:9-step2-4}
\|u\|_{W(I)} \lesssim \|u\|_{M(I)}^{\frac{1}{2}}\|u\|_{Z(I)}^{\frac{1}{2}}
\end{equation}
for any $u \in M(I)$. Thus, by \eqref{es:9-step2-2}, \eqref{es:9-step2-3}, and \eqref{es:9-step2-4},
\begin{equation}\label{v is small in z}
\begin{aligned}
\|v\|_{Z(I^{2})}
\geq \|v\|_{W(I^{2})}^{2}\dfrac{1}{C\|v\|_{M(I^{2})}}
\geq \left(\dfrac{1}{16}\eta\right)^{2}\dfrac{1}{3C\mathcal{E}^{1/2}}
\geq C\eta^{2}\mathcal{E}^{-\frac{1}{2}}.
\end{aligned}
\end{equation}

Now define
\begin{equation}\label{df:v_av}
v_{av}(t,x) = \dfrac{1}{L}\int_{B_{0}(2)} v(t, x+ry)\chi (y) dy,
\end{equation}
where $\chi$ is the cutoff used in \eqref{localize mass}, $r = \eta^{n + 5}|I|^{\frac{1}{4}}$, and $L = \int \chi dx$.
We claim that
\begin{equation}\label{additional regurality}
\|v - v_{av}\|_{Z(I^{2})} \lesssim \mathcal{E}^{\frac{n+4}{2(n-4)}}\eta^{\frac{2n + 10}{n + 4}}.
\end{equation}

Let $\tau_{k}g(x) = g(x + k)$.
We first prove this for $5 \leq n \leq 12$. By \eqref{lplq for H} and H\"older,
\begin{align*}
\|v - \tau_{k}v\|_{L^{\infty}(I^{2}, L^{\frac{2(n+4)}{n-4}})}
&\leq \sup_{t \in I^{2}} \|\int_{t_{-}}^{t_{0}}e^{i(t-s)H}\left( f(u(s)) - f(\tau_{k}u(s)) \right) ds\|_{L^{\frac{2(n+4)}{n-4}}} \\
&\leq C \sup_{t \in I^{2}} \int_{t_{-}}^{t_{0}}|t - s|^{-\frac{2n}{n+4}}\|f(u(s)) - f(\tau_{k}u(s))\|_{L^{\frac{2(n+4)}{n+12}}}ds \\
&\leq C |I|^{-\frac{n+4}{n-4}}
\|\left( |u|^{\frac{8}{n-4}} + |\tau_{k}u|^{\frac{8}{n-4}} \right) |u - \tau_{k}u|\|_{L^{\infty}([t_{-}, t_{0}], L^{\frac{2(n+4)}{n+12}})} \\
&\leq C|I|^{-\frac{n+4}{n-4}} \|u\|_{L^{\infty}_{t}L^{2^{\sharp}}_{x}}^{\frac{8}{n-4}}
\|u - \tau_{k}u\|_{L^{\infty}_{t}L^{2^{\sharp}}_{x}}^{\frac{12 - n}{n+4}}
\|u - \tau_{k}u\|_{L^{\infty}_{t}L^{\frac{2n}{n-2}}_{x}}^{\frac{2(n-4)}{n+4}} \\
&\leq C|I|^{-\frac{n+4}{n-4}}\mathcal{E}^{\frac{-n^{2} + 24n -16}{2(n^{2}-16)}}
\|u - \tau_{k}u\|_{L^{\infty}_{t}L^{\frac{2n}{n-2}}_{x}}^{\frac{2(n-4)}{n+4}}.
\end{align*}
Here we used $\|u\|_{L^{\infty}_{t}L^{2^{\sharp}}_{x}} \leq \mathcal{E}^{\frac{1}{2}}$.
Also, by Sobolev,
\begin{equation}\label{es:9-step2-5}
\|u - \tau_{k}u \|_{L^{\infty}(I, L^{\frac{2n}{n-2}})} \leq |k|\|\nabla u\|_{L^{\infty}(I, L^{\frac{2n}{n-2}})} \leq C|k|\mathcal{E}^{\frac{1}{2}}.
\end{equation}
Hence \eqref{df:v_av} and \eqref{es:9-step2-5} yield
\begin{equation}\label{es:9-step2-6}
\begin{aligned}
\|v - v_{av}\|_{Z(I^{2})}
&\leq \|\dfrac{1}{L} \int_{B_{0}(2)} \left( v(t, x) - v(t, x + ry)\right) \chi (y) dy \|_{Z(I^{2})} \\
&\leq \dfrac{1}{L} \int_{B_{0}(2)} \chi (y) \|v - \tau_{ry}v\|_{Z(I^{2})} dy \\
&\leq C|I|^{\frac{n-4}{2(n+4)}} \int_{B_{0}(2)} \chi (y) \|v - \tau_{ry}v\|_{L^{\infty}(I^{2}, L^{\frac{2(n+4)}{n-4}})}dy \\
&\leq C|I|^{-\frac{n-4}{2(n+4)}} \mathcal{E}^{\frac{n+4}{2(n-4)}}r^{\frac{2(n+4)}{n+4}}\int_{B_{0}(2)}  \chi (y) y dy \\
&\leq C\left( r|I|^{-\frac{1}{4}}\right)^{\frac{2(n-4)}{n+4}}\mathcal{E}^{\frac{n+4}{2(n-4)}}.
\end{aligned}
\end{equation}
Since $\eta < 1$, \eqref{additional regurality} follows for $5 \leq n \leq 12$.

Next suppose $n \geq 13$. We first estimate the gradient of $v$.
Using \eqref{lplq for H}, \eqref{eq1}, and \eqref{eq2},
\begin{equation}\label{es:9-step2-7}
\begin{aligned}
\|\nabla v\|_{L^{\infty}(I^{2}, L^{\frac{2n}{n-6}})}
&\leq \int_{t_{-}}^{t_{0}}\|\nabla e^{i(t-s)H}f(u(s))\|_{L^{\infty}(I^{2}, L^{\frac{2n}{n-6}})} ds \\
&\lesssim \int_{t_{-}}^{t_{0}}\|e^{i(t-s)H}H^{1/4}f(u)\|_{L^{\infty}(I^{2}, L^{\frac{2n}{n-6}})} ds \\
&\lesssim \int_{t_{-}}^{t_{0}}|t-s|^{-\frac{3}{2}} \|H^{1/4}f(u)\|_{L^{\infty}(I^{2}, L^{\frac{2n}{n+6}})} ds \\
&\lesssim \int_{t_{-}}^{t_{0}}|t-s|^{-\frac{3}{2}} \|\nabla f(u)\|_{L^{\infty}(I^{2}, L^{\frac{2n}{n+6}})} ds \\
&\lesssim |I|^{-\frac{1}{2}}\||u|^{\frac{8}{n-4}}\|_{L^{\infty}(I^{2}, L^{\frac{n}{4}})}\|\nabla u \|_{L^{\infty}(I^{2}, L^{\frac{2n}{n-2}})} \\
&\leq C|I|^{-\frac{1}{2}}\mathcal{E}^{\frac{n+4}{2(n-4)}}.
\end{aligned}
\end{equation}
Since \eqref{es:9-step2-0} and Sobolev imply $\|\nabla v\|_{L^{\infty}(I^{2}, L^{\frac{2n}{n-2}})} \lesssim \|\Delta v \|_{L^{\infty}(I^{2}, L^{2})} \leq C\mathcal{E}^{1/2}$, interpolation gives
\begin{equation}\label{es:9-step2-8}
\|\nabla v\|_{L^{\infty}(I^{2}, L^{2^{\sharp}})} \leq \|\nabla v\|_{L^{\infty}(I^{2}, L^{\frac{2n}{n-2}})}^{\frac{1}{2}}\|\nabla v\|_{L^{\infty}(I^{2}, L^{\frac{2n}{n-6}})}^{\frac{1}{2}}
\leq C\mathcal{E}^{\frac{n}{2(n-4)}}|I|^{-\frac{1}{4}}.
\end{equation}
Then Sobolev together with \eqref{es:9-step2-7} and \eqref{es:9-step2-8} yields
\begin{align*}
\|v - \tau_{k}v\|_{L^{\infty}(I^{2}, L^{\frac{2(n+4)}{n-4}})}
&\leq \|v - \tau_{k}v\|_{L^{\infty}(I^{2}, L^{\frac{2n}{n-8}})}^{\frac{n-4}{n+4}} \|v - \tau_{k}v\|_{L^{\infty}(I^{2}, L^{2^{\sharp}})}^{\frac{8}{n+4}} \\
&\lesssim \|\nabla v\|_{L^{\infty}(I^{2}, L^{\frac{2n}{n-6}})}^{\frac{n-4}{n+4}} \|v - \tau_{k}v\|_{L^{\infty}(I^{2}, L^{2^{\sharp}})}^{\frac{8}{n+4}} \\
&\leq C|I|^{-\frac{n-4}{2(n+4)}}\mathcal{E}^{\frac{n^{2} + 8n -16}{2(n^{2}-16)}}\left( |k||I|^{-\frac{1}{4}} \right)^{\frac{8}{n+4}}.
\end{align*}
Arguing as in \eqref{es:9-step2-6}, we obtain
\begin{equation}\label{es:9-step2-9}
\|v - v_{av}\|_{Z(I^{2})} \leq C\mathcal{E}^{\frac{n^{2} + 8n -16}{2(n^{2}-16)}}\left( r|I|^{-\frac{1}{4}} \right)^{\frac{8}{n+4}}.
\end{equation}
Combining \eqref{es:9-step2-6} and \eqref{es:9-step2-9}, we get \eqref{additional regurality} for all $n \geq 5$.

Therefore, by \eqref{v is small in z} and \eqref{additional regurality}, choosing $\eta < \mathcal{E}^{-5n}\eta_{1}$ sufficiently small,
\begin{equation}\label{es:9-step2-10}
\begin{aligned}
\|v_{av}\|_{Z(I^{2})}
&\geq \|v\|_{Z(I^{2})} - \|v-v_{av}\|_{Z(I^{2})} \\
&\geq C\eta^{2}\mathcal{E}^{-\frac{1}{2}} -  C\mathcal{E}^{\frac{n+4}{2(n-4)}}\eta^{\frac{2n + 10}{n + 4}} \\
&\geq C\eta^{2}\mathcal{E}^{-\frac{1}{2}}.
\end{aligned}
\end{equation}
On the other hand, by \eqref{es:9-step2-0} and Sobolev,
\begin{equation}\label{es:9-step2-11}
\|v_{av}\|_{L^{2^{\sharp}}(I^{2}\times \mathbb{R}^{n})} \lesssim \|v_{av}\|_{L^{\infty}(I^{2},\dot{H}^{2})}\left( \int_{I^{2}} dx \right)^{\frac{n-4}{2n}} \lesssim |I|^{\frac{n-4}{2n}}\mathcal{E}^{\frac{1}{2}}.
\end{equation}
Thus, by \eqref{es:9-step2-10} and \eqref{es:9-step2-11},
\[
\|v_{av}\|_{L^{\infty}(I^{2}\times \mathbb{R}^{n})}
\geq \|v_{av}\|_{Z(I^{2})}^{\frac{n+4}{4}}\|v_{av}\|_{L^{2^{\sharp}}(I^{2}\times \mathbb{R}^{n})}^{-\frac{n}{4}}
\geq
C\eta^{\frac{n+4}{2}}|I|^{-\frac{n-4}{8}}\mathcal{E}^{-\frac{n+2}{4}}.
\]
Hence there exists a point $(t_{0}, x_{0}) \in I^{2}\times \mathbb{R}^{n}$ such that
\begin{equation}\label{es:9-step2-12}
\left| \int_{B_{0}(2)} \chi(y) v(t_{0}, x_{0} + ry) dy \right| \geq \dfrac{1}{2}C\eta^{\frac{n+4}{2}}|I|^{-\frac{n-4}{8}}\mathcal{E}^{-\frac{n+2}{4}}.
\end{equation}
By \eqref{es:9-step2-12} and H\"older,
\begin{equation}\label{es:9-step2-13}
\begin{aligned}
M\left(v(t_{0}), B_{x_{0}}(2r)\right)
&= r^{n}\int_{\mathbb{R}^{n}}|v(t_{0}, x + ry)|^{2}\chi^{4}\left(\frac{y}{2}\right)dy \\
&\geq r^{n} \int_{B_{0}(2)} |v(t_{0}, x + ry)|^{2} dy \\
&\geq r^{n}\left( \int_{\mathbb{R}^{n}} \chi^{2}(y) dy\right)^{-1} \left| \int_{B_{0}(2)} \chi(y)v(t_{0}, x_{0} + ry) dy \right|^{2} \\
&\geq C\eta^{K_{1}}|I|\mathcal{E}^{-\frac{n+2}{2}},
\end{aligned}
\end{equation}
where $K_{1} = n^{2} + 6n + 4$.

Since $v$ solves \eqref{linear equation}, \eqref{estimate for local mass 2} yields
\begin{equation}\label{es:9-step2-14}
\left| \partial_{t}\left( M(v(t), B_{x_{0}}(2\eta^{-K_{1}}|I|^{\frac{1}{4}})) \right)^{\frac{3}{4}} \right| \leq C\mathcal{E}^{\frac{3}{4}}\eta^{K_{1}}|I|^{-\frac{1}{4}}
\end{equation}
for any $t \in I$.
Integrating \eqref{es:9-step2-14} over $I$ and using \eqref{es:9-step2-13}, we obtain for any $t \in I$,
\begin{equation}\label{es:9-step2-14-2}
\begin{aligned}
M\left(v(t), B_{x_{0}}(2\eta^{-K_{1}}|I|^{\frac{1}{4}}) \right)
&\geq \left\{ M\left(v(t_{0}), B_{x_{0}}(2\eta^{-K_{1}}|I|^{\frac{1}{4}}) \right)^{\frac{3}{4}} - C\mathcal{E}^{\frac{3}{4}}\eta^{K_{1}}|I|^{\frac{3}{4}} \right\}^{\frac{4}{3}}\\
&\geq  C\eta^{K_{1}}|I|\mathcal{E}^{-\frac{n+2}{2}}.
\end{aligned}
\end{equation}

Since $I$ is unexceptional, \eqref{exceptional} yields
\[
\inf_{\tau \in I} \|u_{-}(\tau)\|_{L^{\frac{2(n+4)}{n-4}}} \left( \int_{I} d\tau \right)^{\frac{n-4}{2(n+4)}} \leq \|u_{-}\|_{Z(I)} \leq C\eta^{K_{2}}.
\]
Hence there exists $\tau \in I$ such that
\begin{equation}\label{es:9-step2-15}
\|u_{-}(\tau)\|_{L^{\frac{2(n+4)}{n-4}}} \leq C|I|^{-\frac{n-4}{2(n+4)}}\eta^{K_{2}}.
\end{equation}
By \eqref{es:9-step2-15} and H\"older, with $R = 2\eta^{-K_{1}}|I|^{\frac{1}{4}}$,
\begin{equation}\label{es:9-step2-16}
\begin{aligned}
M\left(u_{-}(\tau), B_{x_{0}}(2\eta^{-K_{1}}|I|^{\frac{1}{4}})\right)
&\leq \left( \int_{\mathbb{R}^{n}}|u_{-}(\tau)|^{\frac{2(n+4)}{n-4}} dx\right)^{\frac{n-4}{n+4}}
\left( \int_{\mathbb{R}^{n}}\chi^{\frac{n+4}{2}}_{R} dx\right)^{\frac{8}{n+4}} \\
&\leq C\|u_{-}(\tau)\|_{L^{\frac{2(n+4)}{n-4}}}^{2}R^{\frac{8n}{n+4}}\\
&\leq C\eta^{2K_{2}-\frac{8K_{1}n}{n+4}}|I| \leq C\eta^{8K_{1}}|I|.
\end{aligned}
\end{equation}
Applying \eqref{estimate for local mass 2} once again, together with \eqref{es:9-step2-16}, we obtain for all $t \in I$,
\begin{equation}\label{es:9-step2-17}
M\left(u_{-}(t), B_{x_{0}}(2\eta^{-K_{1}}|I|^{\frac{1}{4}})\right) \leq C|I|\eta^{\frac{4}{3}K_{1}}\mathcal{E}.
\end{equation}
Since $u(t_{0}) = u_{-}(t_{0}) + v(t_{0})$, combining \eqref{es:9-step2-14-2} and \eqref{es:9-step2-17},
\begin{align*}
C\eta^{K_{1}}|I|\mathcal{E}^{-\frac{n+2}{2}}
&\lesssim M\left(v(t_{0}), B_{x_{0}}(2\eta^{-K_{1}}|I|^{\frac{1}{4}}) \right) \\
&\lesssim M\left(u(t_{0}), B_{x_{0}}(2\eta^{-K_{1}}|I|^{\frac{1}{4}}) \right) + M\left(u_{-}(t_{0}), B_{x_{0}}(2\eta^{-K_{1}}|I|^{\frac{1}{4}}) \right) \\
&\lesssim M\left(u(t_{0}), B_{x_{0}}(2\eta^{-K_{1}}|I|^{\frac{1}{4}}) \right) + |I|\eta^{\frac{4}{3}K_{1}}\mathcal{E}.
\end{align*}
Thus, choosing $\eta$ sufficiently small,
\begin{equation}\label{es:9-step2-18}
M\left(u(t_{0}), B_{x_{0}}(2\eta^{-K_{1}}|I|^{\frac{1}{4}})\right) \geq C\eta^{K_{1}}|I|\mathcal{E}^{-\frac{n+2}{2}}.
\end{equation}
Finally, \eqref{conclusion of step9-2} follows from \eqref{estimate for local mass 2} and \eqref{es:9-step2-18}.
\end{proof}

\begin{step}
Let $u \in C([t_{-}, t_{+}], H^{2}_{rad})$ be a radial solution to \eqref{duhamel2}, and let $I$ be an unexceptional interval.
Then
\begin{equation}\label{conclusion of step9.3}
\int_{I}\int_{B_{0}(2\eta^{-4K_{1}}|I|^{\frac{1}{4}})}\dfrac{|u(t,x)|^{\frac{2n}{n-4}}}{|x|}dx \geq C\eta^{13K_{1}}\mathcal{E}^{-4n}|I|^{\frac{3}{4}},
\end{equation}
where $K_{1} = n^{2} + 6n +4$, and $C$ depends only on $n$ and $\lambda$.
\end{step}

\begin{proof}
Set $R = 2\eta^{-K_{1}}|I|^{\frac{1}{4}}$. Then \eqref{conclusion of step9-2} and H\"older's inequality give
\begin{equation}\label{es:9-step3-1}
\begin{aligned}
\int_{B_{x_{0}}(R)} |u|^{2^{\sharp}} dx
&\geq M(u, 2\eta^{-K_{1}}|I|^{\frac{1}{4}})^{\frac{n}{n-4}} \left( \int_{B_{x_{0}}(R)} \chi_{R}^{n}(x) dx \right)^{-\frac{4}{n}} \\
&\geq C\eta^{\frac{5K_{1}n}{n-4}}\mathcal{E}^{-\frac{n(n+2)}{2(n-4)}} \geq C\eta^{9K_{1}}\mathcal{E}^{-4n}.
\end{aligned}
\end{equation}
Next we show that $|x_{0}| \leq \eta^{-4K_{1}}|I|^{\frac{1}{4}}$. If not, then there exist at least $\eta^{-3(n-1)K_{1}}/ 4^{n-1}$ pairwise disjoint balls of radius $2\eta^{-K_{1}}|I|^{\frac{1}{4}}$ centered at distance $|x_{0}|$ from the origin.
Hence, using radial symmetry and the facts that $\lambda > 0$ and $V \geq 0$, \eqref{es:9-step3-1} implies for any $t \in I$,
\begin{equation}\label{es:9-step3-2}
\dfrac{2^{\sharp}}{\lambda} E(u(t))
\geq \|u\|_{L^{2^{\sharp}}}^{2^{\sharp}}
\geq \dfrac{1}{4^{n-1}}\eta^{-3(n-1)K_{1}}C\eta^{9K_{1}}\mathcal{E}^{-4n}
\geq C\eta^{-2K_{1}}\mathcal{E}^{-4n}.
\end{equation}
If $\eta$ is sufficiently small, \eqref{es:9-step3-2} contradicts $E(u(t)) \lesssim \mathcal{E}$.
Therefore $|x_{0}| \leq \eta^{-4K_{1}}|I|^{\frac{1}{4}}$, and hence
\[
B_{x_{0}}(2\eta^{-K_{1}}|I|^{\frac{1}{4}}) \subset B_{0}(2\eta^{-4K_{1}}|I|^{\frac{1}{4}}).
\]
Thus, for every $t\in I$,
\begin{equation}\label{es:9-step3-3}
M\left(u(t), B_{0}(2\eta^{-4K_{1}}|I|^{\frac{1}{4}})\right) \geq C\eta^{K_{1}}\mathcal{E}^{-\frac{n+2}{2}}|I|
\end{equation}
provided $\eta < \mathcal{E}^{-5n}\eta_{1}$, where $\eta_{1}$ is sufficiently small depending only on $n$ and $\lambda$.
Now redefine $R = 2\eta^{-4K_{1}}|I|^{\frac{1}{4}}$. Then H\"older's inequality yields
\begin{align*}
M\left(u(t), B_{0}(2\eta^{-4K_{1}}|I|^{\frac{1}{4}})\right)
&\leq \left( \int_{B_{0}(R)} \dfrac{|u|^{2^{\sharp}}}{|x|} \right)^{\frac{n-4}{n}}
\left( \int_{B_{0}(R)} |x|^{-\frac{n-4}{4}}\chi_{R}^{n} \right)^{\frac{4}{n}} \\
&\leq CR^{\frac{5n -4}{n}}\left( \int_{B_{0}(R)} \dfrac{|u|^{2^{\sharp}}}{|x|} \right)^{\frac{n-4}{n}}.
\end{align*}
Combining this with \eqref{es:9-step3-3}, we conclude
\begin{align*}
\int_{B_{0}(R)} \dfrac{|u|^{2^{\sharp}}}{|x|}
\geq C M\left(u(t), B_{0}(2\eta^{-4K_{1}}|I|^{\frac{1}{4}})\right)^{\frac{n}{n-4}} R^{-\frac{5n-4}{n-4}}
\geq C\eta^{13K_{1}}\mathcal{E}^{-4n}|I|^{\frac{3}{4}}.
\end{align*}
\end{proof}

\begin{step}\label{step4}
Let $u \in C([t_{-}, t_{+}], H^{2}_{rad})$ be a radial solution to \eqref{duhamel2}, and let
\[
I = \bigcup_{j_{1} \leq j \leq j_{2}}I_{j}
\]
be a union of consecutive unexceptional intervals for $u$.
Then there exists $j_{1} \leq j_{0} \leq j_{2}$ such that
\begin{equation}\label{conclusion of step9.4}
|I_{j_{0}}| \geq K|I|,
\end{equation}
where $K = C\eta^{100K_{1}}\mathcal{E}^{-20n}$, and $C>0$ depends only on $n$ and $\lambda$.
\end{step}

\begin{proof}
By \eqref{Morawetz} and \eqref{conclusion of step9.3},
\begin{equation}\label{es:9-step4-1}
\begin{aligned}
C\eta^{13K_{1}}\mathcal{E}^{-4n}\sum_{j_{1} \leq j \leq j_{2}}|I_{j}|^{\frac{3}{4}}
&\leq \sum_{j}\int_{I_{j}}\int_{B_{0}(2\eta^{-4K_{1}}|I_{j}|^{\frac{1}{4}})} \dfrac{|u|^{2^{\sharp}}}{|x|} dx \\
&\leq \int_{I}\int_{B_{0}(2\eta^{-4K_{1}}|I|^{\frac{1}{4}})}\dfrac{|u|^{2^{\sharp}}}{|x|} dx \\
&\leq C\eta^{-12K_{1}}\left( \mathcal{E} + \mathcal{E}^{\frac{n}{n-4}}\right)|I|^{\frac{3}{4}}.
\end{aligned}
\end{equation}
Set $\tilde{K} = C\mathcal{E}^{5n}\eta^{-25K_{1}}$. Then \eqref{es:9-step4-1} implies
\begin{equation}\label{es:9-step4-2}
\begin{aligned}
\left( \max_{j}|I_{j}| \right)^{-\frac{1}{4}}\sum_{j}|I_{j}|
\leq \sum_{j}|I_{j}|^{\frac{3}{4}}
\leq \tilde{K}|I|^{\frac{3}{4}}
\leq \tilde{K}\left( \sum_{j}|I_{j}| \right)^{\frac{3}{4}}.
\end{aligned}
\end{equation}
It is immediate from \eqref{es:9-step4-2} that \eqref{conclusion of step9.4} follows.
\end{proof}

The following proposition is due to Killip, Visan, and Zhang \cite{Killip}.

\begin{proposition}\label{prop9-2}
Assume that an interval $I$ is covered by finitely many subintervals $I_{1},...,I_{N}$.
Suppose that for every collection of consecutive intervals $\{ I_{j} : j \in \mathcal{J}\}$, there exists $j_{*} \in \mathcal{J}$ such that
\[
 |I_{j_{*}}| \geq K\big|\bigcup_{j \in \mathcal{J}} I_{j}\big|
\]
for some $K > 0$.
Then there exist $M \geq \text{ln}(N) / \text{ln}(2K^{-1})$ and distinct indices $j_{1},...j_{M}$ such that for any $l < k$,
\begin{equation}\label{conclusion of prop9-2}
\begin{aligned}
&|I_{j_{1}}| \geq 2|I_{j_{2}}| \geq \cdots \geq 2^{M-1}|I_{j_{M}}|,\\
&\text{dist}(I_{j_{l}},\  I_{j_{k}})  \leq (2K)^{-1}|I_{j_{l}}|.
\end{aligned}
\end{equation}
\end{proposition}

\begin{step}\label{step5}
Let $u \in C([t_{-}, t_{+}], H^{2}_{rad})$ be a radial solution to \eqref{duhamel2}, and let $I_{j_{1}},...,I_{j_{M}}$ be pairwise disjoint unexceptional intervals satisfying \eqref{conclusion of prop9-2} with
\[
K = C\eta^{100K_{1}}\mathcal{E}^{-20n}.
\]
Then
\[
M \leq C\mathcal{E}\eta^{-5000n^{2}}\text{ln}(1/\eta),
\]
where $C$ depends only on $n$ and $\lambda$.
\end{step}

\begin{proof}
Take $t_{*}\in I_{j_{M}}$. Combining \eqref{es:9-step3-3}, \eqref{estimate for local mass 2}, and \eqref{conclusion of prop9-2}, we obtain for every $1 \leq k \leq M$,
\begin{equation}\label{es:9-step5-1}
\begin{aligned}
M\left(u(t_{*}), B_{0}(2\eta^{-101K_{1}}|I_{j_{k}}|^{\frac{1}{4}})\right)
\geq C\eta^{K_{1}}\mathcal{E}^{-\frac{n+2}{2}}|I_{j_{k}}|.
\end{aligned}
\end{equation}
Moreover, \eqref{estimate for local mass 1} gives
\begin{equation}\label{es:9-step5-2}
M\left( u(t_{*}), B_{0}(R) \right) \leq C\mathcal{E}R^{4}.
\end{equation}
Consider the annuli
\[
A_{k} = \{ x: \eta^{K_{1}}|I_{j_{k}}|^{\frac{1}{4}} \leq |x| \leq 2\eta^{-101K_{1}}|I_{j_{k}}|^{\frac{1}{4}} \}.
\]
By \eqref{es:9-step5-1} and \eqref{es:9-step5-2},
\begin{equation}\label{es:9-step5-3}
\begin{aligned}
\int_{A_{k}}|u(t_{*}, x)|^{2}dx
&\geq  M\left(u(t_{*}), B_{0}(2\eta^{-101K_{1}}|I_{j_{k}}|^{\frac{1}{4}})\right) - \int_{|x| \leq \eta^{K_{1}}|I_{j_{k}}|^{\frac{1}{4}}}|u(t_{*},x)|^{2}dx \\
&\geq C\eta^{K_{1}}\mathcal{E}^{-\frac{n+2}{2}}|I_{j_{k}}| - C\mathcal{E}^{4K_{1}}|I_{j_{k}}| \geq C\eta^{K_{1}}\mathcal{E}^{-\frac{n+2}{2}}|I_{j_{k}}|.
\end{aligned}
\end{equation}
Hence, by H\"older,
\begin{equation}\label{es:9-step5-4}
\begin{aligned}
\int_{A_{k}}|u(t_{*},x)|^{\frac{2n}{n-4}}dx \geq C\eta^{K_{3}},
\end{aligned}
\end{equation}
where $K_{3} = 5000n^{2}$.

Now set
\[
l = 400K_{1}\text{ln}\left(\frac{1}{\eta}\right)/\text{ln}2.
\]
Then for any $k$,
\begin{equation}\label{es:9-step5-5}
\eta^{-101K_{1}}|I_{j_{1 + (k+1)l}}|^{\frac{1}{4}} \leq \eta^{K_{1}}|I_{j_{1 + kl}}|^{\frac{1}{4}}.
\end{equation}
Indeed, by \eqref{conclusion of prop9-2},
\[
\eta^{-101K_{1}}|I_{j_{1 + (k+1)l}}|^{\frac{1}{4}} \leq \eta^{-101K_{1}} 2^{-\frac{l}{4}}|I_{j_{1 + kl}}|^{\frac{1}{4}},
\]
and
\[
\text{ln}(\eta^{-101K_{1}}2^{-\frac{l}{4}}) = 101K_{1}\text{ln}\left(\frac{1}{\eta}\right) - \frac{l}{4}\text{ln}2 = \text{ln}(\eta^{K_{1}}).
\]
Thus \eqref{es:9-step5-5} holds, and therefore the sets $A_{1 + kl}$ are pairwise disjoint.
Consequently,
\[
 \mathcal{E} \geq C\sum_{kl \leq M-1} \int_{A_{1 + kl}}|u(t_{*}, x)|^{2^{\sharp}}dx \geq MC\eta^{K_{3}}l^{-1},
\]
which proves the claim.
\end{proof}

\begin{proof}[Proof of Proposition \ref{prepare main thm}]
By \eqref{bound of exceptional}, the number $N_{e}$ of exceptional intervals is bounded.
Now let $J = \bigcup_{j \in \mathcal{J}}I_{j}$ be a cluster of unexceptional intervals lying between exceptional intervals.
By Step \ref{step4}, Step \ref{step5}, and Proposition \ref{prop9-2}, choosing $\eta_{1}$ sufficiently small depending only on $n$ and $\lambda$, and taking $\eta \leq \mathcal{E}^{-5n}\eta_{1}$, we obtain
\[
 |\mathcal{J}| \leq C\mathcal{E}\eta^{-5000n^{2}}\text{ln}(1/\eta) \leq C\eta^{-C\eta^{-5000n^{2}}}.
\]
There are only $N_{e}$ such collections of unexceptional intervals, and hence the conclusion follows.
\end{proof}

\section{Scattering for the critical equation}

In this section we prove Theorem \ref{th:scattering}.

\begin{assumption}\label{as: linear scattering}
Let $H = H_{0} + V$, where $H_{0} = \Delta^{2}$ and $V$ is a real-valued potential.
\begin{enumerate}
\item[(H1)]
The wave operators
\[
W_{\pm} = s-\lim_{t \rightarrow \pm \infty} e^{itH}e^{-itH_{0}}
\]
exist on $L^{2}$.
\item[(H2)]
The inverse wave operators
\[
\Omega_{\pm} = s-\lim_{t \rightarrow \pm \infty} e^{itH_{0}}e^{-itH}P_{ac}(H)
\]
exist on $L^{2}$.
\end{enumerate}
\end{assumption}

Let $\mathbb{B}_{\infty}(X, Y)$ denote the space of compact operators from $X$ to $Y$, and write $\mathbb{B}_{\infty}(X) = \mathbb{B}_{\infty}(X, X)$.
Also, define $L^{\infty, 0}(\mathbb{R}^{n})$ to be the set of all $f \in L^{\infty}$ such that $f(x)\to 0$ as $|x|\to\infty$.

\begin{lemma}\label{lemma:scattering}
Let $V \in L^{\infty, 0}(\mathbb{R}^{n})$. Then
\[
[\psi(H_{0}) - \psi(H)]\varphi(H_{0}) \in \mathbb{B}_{\infty}(L^{2}, H^{2})
\]
for any $\psi, \varphi \in C^{\infty}_{0}(\mathbb{R})$.
\end{lemma}

\begin{proof}
It suffices to show that
\[
\langle H_{0} \rangle^{1/2} [ \psi(H_{0}) - \psi(H)]\varphi(H_{0}) \in \mathbb{B}_{\infty}(L^{2}).
\]
By the Helffer--Sj\"ostrand formula,
\begin{equation}\label{lemma:sca:1}
\begin{aligned}
\langle H_{0} \rangle^{1/2} &[ \psi(H_{0}) - \psi(H)]\varphi(H_{0}) \\
&= \dfrac{1}{2\pi i}\int_{\mathbb{C}}\dfrac{\partial \tilde{\varphi} }{\partial \bar{z} } (z) \langle H_{0} \rangle^{1/2} [ (H_{0} - z)^{-1} - (H - z)^{-1} ]\varphi(H_{0})dz d\bar{z},
\end{aligned}
\end{equation}
where $\tilde{\varphi} \in C_{0}^{\infty}(\mathbb{R}^{2})$ is an almost-analytic extension of $\varphi$ satisfying
\[
\partial_{\bar{z}}\tilde{\varphi}(z) = O(\langle z \rangle^{-N}|\text{Im}z|^{N})
\]
for any $N \geq 1$.

We claim that
\[
\langle H_{0} \rangle^{1/2} [ (H_{0} - z)^{-1} - (H - z)^{-1} ] \in \mathbb{B}_{\infty}(L^2).
\]
Indeed, by the resolvent identity,
\begin{align*}
&\langle H_{0} \rangle^{1/2} [ (H_{0} - z)^{-1} - (H - z)^{-1} ] \\
&= \langle H_{0} \rangle^{1/2} (H - z)^{-1}V(H_{0} - z)^{-1} \\
&= \langle H_{0} \rangle^{1/2} \langle H \rangle^{-1/2} \langle H \rangle^{1/2} (H - z)^{-1}V(H_{0} - z)^{-1}.
\end{align*}
Now $\langle H_{0} \rangle^{1/2} \langle H \rangle^{-1/2}$ and $\langle H \rangle^{1/2} (H - z)^{-1}$ are bounded on $L^{2}$, and since $V \in L^{\infty,0}(\mathbb{R}^{n})$, the operator
\[
V(H_{0} - z)^{-1} \in \mathbb{B}_{\infty}(L^{2})
\]
is compact. Therefore the integrand in \eqref{lemma:sca:1} is compact. Also, since $\varphi(H_{0}) \in \mathbb{B}(L^{2})$, the full integrand is compact.

Writing $\|\cdot\| = \|\cdot\|_{L^{2} \rightarrow L^{2}}$, the spectral theorem and Proposition \ref{prop:inh:sobolev} imply
\begin{equation}
\begin{aligned}
&\|\langle H_{0} \rangle^{1/2} [ (H_{0} - z)^{-1} - (H - z)^{-1} ]\varphi(H_{0})\|\\
&\leq \|\langle H_{0} \rangle^{1/2}(H_{0} - z)^{-1}\varphi(H_{0})\|
+ \|\langle H_{0} \rangle^{1/2}(H - z)^{-1}\varphi(H_{0})\| \\
&\lesssim \|(H_{0} - z)^{-1}\langle H_{0} \rangle^{1/2}\varphi(H_{0})\|
+ \|(H - z)^{-1}\langle H \rangle^{1/2}\varphi(H_{0})\| \\
&\lesssim \|(H_{0} - z)^{-1}\|\|\langle H_{0} \rangle^{1/2}\varphi(H_{0})\|
+ \|(H - z)^{-1}\|\|\langle H_{0} \rangle^{1/2}\varphi(H_{0})\| \\
&\lesssim \sup_{\lambda \in \sigma(H_{0})}|\lambda - z|^{-1}\|\langle H_{0} \rangle^{1/2}\varphi(H_{0})\|
+ \sup_{\lambda \in \sigma(H)}|\lambda - z|^{-1}\|\langle H_{0} \rangle^{1/2}\varphi(H_{0})\| \\
&\leq C|\text{Im}z|^{-1}.
\end{aligned}
\end{equation}
Thus the integral in \eqref{lemma:sca:1} converges in operator norm, and the result follows.
\end{proof}

\begin{proposition}\label{scattering-1}
Assume $V \in L^{\infty, 0}(\mathbb{R}^{n})$. Then the following hold:
\begin{enumerate}
\item If (H1) holds, then
\[
W_{\pm} = s-\lim_{t \rightarrow \pm \infty} e^{itH}e^{-itH_{0}}
\]
exist on $H^{2}$.
\item If (H2) holds, then
\[
\Omega_{\pm} = s-\lim_{t \rightarrow \pm \infty} e^{itH_{0}}e^{-itH}P_{ac}(H)
\]
exist on $H^{2}$.
\end{enumerate}
\end{proposition}

\begin{proof}
Let $u_{0} \in H^{2}$ and set $W(t) = e^{itH}e^{-itH_{0}}$. By Proposition \ref{prop:inh:sobolev},
\[
\|\langle \Delta \rangle W(t)u_{0} \|_{L^{2}} \lesssim \|e^{itH}\langle H \rangle^{1/2} e^{-itH_{0}} u_{0}\|_{L^{2}} \lesssim \|u_{0}\|_{H^{2}},
\]
so $W(t)$ is uniformly bounded on $H^{2}$ with respect to $t$.

To prove existence, it suffices to show that for any sequence $t_{n} \to \infty$, the sequence $\{W(t_{n})u_{0}\}$ is Cauchy in $H^{2}$.
Fix $\varepsilon > 0$, and choose $\varphi \in C^{\infty}_{0}$ such that $\varphi \equiv 1$ near the origin. Let $\varphi_{R}(\lambda) = \varphi(\lambda / R)$. Then
\[
\|u_{0} - \varphi_{R}(H_{0})u_{0}\|_{H^{2}}^{2}
= \int_{\sigma(H_{0})} |1 - \varphi(\lambda / R)|^{2} d\|E_{H_{0}}(\lambda)u_{0}\|^{2}_{H^{2}} \rightarrow 0,\qquad R \rightarrow \infty
\]
by the dominated convergence theorem.
In particular, for some $R \geq 1$,
\[
\sup_{n, m}\|(W(t_{n}) - W(t_{m}))(u_{0} - \varphi_{R}(H_{0})u_{0})\|_{H^{2}}
\lesssim \|u_{0} - \varphi_{R}(H_{0})u_{0}\|_{H^{2}} < \varepsilon.
\]
Hence we may replace $u_{0}$ by $u_{R} = \varphi_{R}(H_{0})u_{0}$.

Now choose $\psi \in C^{\infty}_{0}$ such that $\psi\varphi \equiv \varphi$. Then
\begin{align*}
[W(t_{n}) - W(t_{m})]u_{R}
= \psi_{R}(H)[W(t_{n}) - W(t_{m})]u_{R} + (1 - \psi_{R})(H)[W(t_{n}) - W(t_{m})]u_{R}.
\end{align*}
Since $\|\langle H_{0} \rangle^{1/2} \psi_{R}(H)\|_{L^{2} \rightarrow L^{2}} \lesssim R$, assumption (H1) implies
\[
\|\psi_{R}(H)[W(t_{n}) - W(t_{m})]u_{R}\|_{H^{2}}
\lesssim R\|[W(t_{n}) - W(t_{m})]u_{R}\|_{L^{2}} \rightarrow 0, \qquad m,n \rightarrow \infty.
\]
Moreover,
\[
(1 - \psi_{R})(H)\varphi_{R}(H_{0}) = [ \psi_{R}(H_{0}) - \psi_{R}(H) ]\varphi_{R}(H_{0}),
\]
and for any $f \in L^{2}$, one has weak convergence $e^{itH_{0}}f \rightharpoonup 0$ in $L^{2}$ as $t \to \pm \infty$. Therefore, by Lemma \ref{lemma:scattering},
\begin{align*}
\|(1 - \psi_{R})(H)[W(t_{n}) - W(t_{m})]u_{R}\|_{H^{2}}
\leq &\|[\psi_{R}(H_{0}) - \psi_{R}(H)]e^{-it_{n}H_{0}}u_{R}\|_{H^{2}} \\
&+ \|[\psi_{R}(H_{0}) - \psi_{R}(H)]e^{-it_{m}H_{0}}u_{R}\|_{H^{2}} \\
\rightarrow 0, \qquad &m,n \rightarrow \infty.
\end{align*}
Hence $\{W(t_{n})u_{R}\}$ is Cauchy in $H^{2}$.
The proof for $\Omega_{\pm}$ is the same.
\end{proof}

\begin{remark}\label{remark:scattering}
It is known from \cite{Kuroda} that if $V \in L^{\infty}$ satisfies
\[
|V(x)| \lesssim \langle x \rangle^{- \mu}
\]
for some $\mu > 1$, then Assumption \ref{as: linear scattering}, namely (H1) and (H2), holds.
\end{remark}

\begin{proposition}\label{scattering-2}
Let $V$ be a radial real-valued potential satisfying Assumptions \ref{as:wave}, \ref{as:mora}, and \ref{as:equ}. Let $n \geq 5$, and let $u \in C(\mathbb{R}, H^{2}_{rad})$ be any radial solution to \eqref{eq:NL4S} with $p = 2^{\sharp}-1$ and $\lambda > 0$.
Then there exist $u^{\pm} \in H^{2}_{rad}$ such that
\begin{equation}\label{conclusion of prop10-1}
\|u(t) - e^{itH}u^{\pm}\|_{H^{2}} \rightarrow 0,\qquad t \rightarrow \pm \infty.
\end{equation}
Moreover, $u^{\pm}$ are unique and are determined by \eqref{conclusion of prop10-1}.
Thus we may define two maps $S_{\pm} : u(0) \mapsto u^{\pm}$ from $H^{2}_{rad}$ to $H^{2}_{rad}$.
These maps are continuous on $H^{2}_{rad}$.
Furthermore, if
\[
 \|e^{itH}u^{\pm} - e^{it\Delta^{2}}u^{\pm}_{*}\|_{H^{2}} \rightarrow 0 , \qquad t \rightarrow \pm \infty,
\]
then
\begin{equation}\label{conclusion of prop10-2}
\begin{aligned}
M(u_{0}) &= M(u^{\pm}), \\
2E(u_{0}) &= \|u^{\pm}_{*}\|^{2}_{\dot{H}^{2}}.
\end{aligned}
\end{equation}
\end{proposition}

\begin{proof}
By time reversal symmetry, it suffices to prove the statement for $u^{+}$.
By Proposition \ref{prepare main thm}, the $Z$-norm of $u$ on $\mathbb{R}_{+}$ is finite. Together with Proposition \ref{criterion}, this implies that the $W(\mathbb{R}_{+})$-norm of $u$ is finite.

Define
\[
v(t) = e^{-itH}u(t), \qquad t > 0.
\]
It suffices to show that $v(t)$ converges in $H^{2}$ as $t \to \infty$.
By Duhamel's formula,
\begin{equation}\label{es:10-prop1-1}
v(t_{1}) - v(t_{2}) = i\lambda \int_{t_{1}}^{t_{2}} e^{-isH}|u(s)|^{\frac{8}{n-4}}u(s)ds.
\end{equation}
By \eqref{strichartz for scattering},
\[
 \| \int_{\mathbb{R}}e^{-isH}|u(s)|^{\frac{8}{n-4}}u(s)ds \|_{\dot{H}^{2}}
 \leq C\||u|^{\frac{8}{n-4}}u \|_{N(\mathbb{R})}
 \leq C\|u\|^{\frac{n+4}{n-4}}_{W(\mathbb{R})} < \infty.
\]
Hence the right-hand side of \eqref{es:10-prop1-1} tends to $0$ in $\dot{H}^{2}$ as $t_{1}, t_{2} \to \infty$.
Also, by Proposition \ref{criterion}, $u\in L^{\frac{2(n+4)}{n}}(\mathbb{R}_{+}\times\mathbb{R}^{n})$, and therefore \eqref{standard} yields
\[
 \| \int_{\mathbb{R}}e^{-isH}|u(s)|^{\frac{8}{n-4}}u(s)ds \|_{L^{2}}
 \leq C\|u\|_{Z(\mathbb{R})}^{\frac{8}{n-4}}\|u\|_{L^{\frac{2(n+4)}{n}}(\mathbb{R}\times\mathbb{R}^{n})} < \infty.
\]
Similarly, the right-hand side of \eqref{es:10-prop1-1} tends to $0$ in $L^{2}$ as $t_{1}, t_{2}\to\infty$.
Therefore $v(t)$ is Cauchy in $H^{2}$, so there exists $u^{+} \in H^{2}$ such that \eqref{conclusion of prop10-1} holds.
Moreover,
\[
 u^{+}  = u_{0} + i\lambda \int_{0}^{\infty} e^{-isH}|u(s)|^{\frac{8}{n-4}}u(s)ds.
\]

We next prove the continuity of the map $u_{0} \mapsto u^{+}$. Continuity in $\dot{H}^{2}$ follows from Proposition \ref{prop:long time} in the case $e=0$. Continuity in $L^{2}$ follows, as in Proposition \ref{criterion}, from the boundedness of the $Z$-norm in Proposition \ref{prepare main thm} together with \eqref{standard}.

We now prove \eqref{conclusion of prop10-2}. The first identity follows from conservation of mass and the fact that $v(t)$ converges in $L^{2}$ as $t\to\infty$.
For the second identity, since $u \in Z(\mathbb{R}_{+})$, there exists a sequence $t_{k}$ such that
\[
\|u(t_{k})\|_{L^{\frac{2(n+4)}{n-4}}} \rightarrow 0 \qquad (t_{k} \rightarrow \infty).
\]
Together with conservation of mass, this implies
\[
 \|u(t_{k})\|_{L^{2^{\sharp}}}^{2^{\sharp}}
 \leq \|u(t_{k})\|_{L^{\frac{2(n+4)}{n-4}}}^{\frac{n+4}{n-4}}\|u_{0}\|_{L^{2}} \rightarrow 0,\qquad t_{k} \rightarrow \infty.
\]
Also,
\[
 \int_{\mathbb{R}^{n}} V(x) |u(x)|^{2} dx \leq \|V\|_{L^{\frac{n}{4}, \infty}}\|u(t_{k})\|_{L^{2^{\sharp}}}^{2} \rightarrow 0,\qquad t_{k} \rightarrow \infty.
\]
Hence, writing $w(t) = e^{itH}u^{+}$,
\begin{align*}
2E(u_{0})
&= 2E(u(t_{k})) \\
&= \|u(t_{k})\|^{2}_{\dot{H}^{2}} + o(1) \\
&= \|w(t_{k})\|^{2}_{\dot{H}^{2}} + o(1) \\
&= \|e^{it_{k}\Delta^{2}}u_{*}^{+}\|^{2}_{\dot{H}^{2}} + o(1).
\end{align*}
Letting $k \to \infty$, we obtain \eqref{conclusion of prop10-2}.
\end{proof}

\begin{proof}[Proof of Theorem \ref{th:scattering}]
Under Assumptions \ref{as:wave} and \ref{as:equ}, note that $P_{ac}(H)$ is the identity.
Therefore, by Proposition \ref{scattering-1}, Proposition \ref{scattering-2}, and Remark \ref{remark:scattering}, there exist $u^{\pm}, \tilde{u}^{\pm} \in H^{2}_{rad}$ such that as $t \to \pm \infty$,
\begin{align*}
 \|u(t) - e^{itH}\tilde{u}^{\pm}\|_{H^{2}} &\rightarrow 0, \\
 \|e^{itH}\tilde{u}^{\pm} - e^{it\Delta^{2}}u^{\pm}\|_{H^{2}} &\rightarrow 0.
\end{align*}
Hence
\[
 \|u(t) - e^{it\Delta^{2}}u^{\pm} \|_{H^{2}}
 \leq  \|u(t) - e^{itH}\tilde{u}^{\pm}\|_{H^{2}} + \|e^{itH}\tilde{u}^{\pm} - e^{it\Delta^{2}}u^{\pm}\|_{H^{2}}
 \rightarrow 0.
\]
\end{proof}

\begin{proposition}
Let $V$ be a radial real-valued potential satisfying Assumptions \ref{as:wave}, \ref{as:mora}, and \ref{as:equ}. Let $n \geq 5$.
For any $u^{+} \in H^{2}_{rad}$ (resp. $u^{-} \in H^{2}_{rad}$), there exists a solution $u \in C(\mathbb{R}, H^{2}_{rad})$ to \eqref{eq:NL4S} with $\lambda > 0$ and $p = 2^{\sharp} - 1$ such that \eqref{conclusion of prop10-1} holds.
In particular, $S_{\pm}$ are homeomorphisms on $H^{2}_{rad}$.
\end{proposition}

\begin{proof}
By time reversal symmetry, it suffices to prove the statement for $u^{+}$.
Let $w(t) = e^{itH}u^{+}$. Then \eqref{strichartz for scattering} implies that $w \in W(\mathbb{R})$.
Hence, for any $\delta > 0$, there exists $T_{\delta}$ such that
\[
\|w\|_{W([T_{\delta}, \infty))} < \delta.
\]
For any $u \in W([T_{\delta}, \infty))$, define
\begin{equation}\label{prop10-4:1}
\Phi (u)(t) = w(t) -i \lambda \int_{t}^{\infty} e^{i(t- s)H}|u(s)|^{\frac{8}{n-4}}u(s) ds.
\end{equation}
Consider the complete metric space
\begin{align*}
X_{T_{\delta}} = &\left\{ u \in W([T_{\delta}, \infty)) \cap L^{\frac{2(n+4)}{n}}([T_{\delta}, \infty), L^{\frac{2(n+4)}{n}}) ; \ \|u\|_{W([T_{\delta}, \infty))} \leq C\delta,\right. \\
&\|u\|_{L^{\frac{2(n+4)}{n}}([T_{\delta}, \infty), L^{\frac{2(n+4)}{n}})}\left. \leq C\|u^{+}\|_{L^{2}}\right\},
\end{align*}
equipped with the metric induced by
$L^{\frac{2(n+4)}{n}}([T_{\delta}, \infty), L^{\frac{2(n+4)}{n}})$.
Then $\Phi$ is a contraction on $X_{T_{\delta}}$, so there exists $u$ such that $\Phi(u)=u$.

By the Strichartz estimate, $u \in C([T_{\delta}, +\infty), H^{2}) \cap M([T_{\delta}, +\infty))$.
Moreover,
\[
 u(T_{\delta} + t) = e^{itH}u(T_{\delta}) + i\lambda\int_{T_{\delta}}^{T_{\delta} + t} e^{i(t + T_{\delta} - s)}|u(s)|^{\frac{8}{n-4}}u(s) ds,
\]
so $u$ is a solution of \eqref{eq:NL4S} with $p = 2^{\sharp} - 1$ on $[T_{\delta}, +\infty)$.
Therefore, by Theorem \ref{th:global}, $u$ extends to all of $\mathbb{R}$.
From \eqref{prop10-4:1}, one proves \eqref{conclusion of prop10-1} in the same way as in Proposition \ref{scattering-2}. The continuity of $S_{\pm}$ is proved similarly.
\end{proof}

\appendix
\section{Appendix}

\begin{app}[Abstract bootstrap principle,\ \cite{Tao}]\label{bootstrap}
Let $I$ be a time interval, and suppose that for each $t \in I$ there are two statements: a ``hypothesis'' $\textbf{H}(t)$ and a ``conclusion'' $\textbf{C}(t)$.
Assume that the following four properties hold.
\begin{enumerate}
\item[(a)] If $\textbf{H}(t)$ holds for some $t \in I$, then $\textbf{C}(t)$ also holds for that $t$.
\item[(b)] If $\textbf{C}(t)$ holds for some $t \in I$, then $\textbf{H}(t')$ holds for all $t'$ in a neighborhood of $t$.
\item[(d)] If a sequence $t_{n} \in I$ converges to $t \in I$ and $\textbf{C}(t_{n})$ holds for all $n$, then $\textbf{C}(t)$ also holds.
\item[(e)] $\textbf{H}(t)$ holds for at least one point $t \in I$.
\end{enumerate}
Then $\textbf{C}(t)$ holds for every $t \in I$.
\end{app}

\begin{app}[Product rule,\ \cite{Christ}]\label{Product rule}
Let $s \in (0, 1]$, and let $1 < r, p_{1}, p_{2}, q_{1}, q_{2} < \infty$ satisfy
\[
\frac{1}{r} = \frac{1}{p_{i}} + \frac{1}{q_{i}},\qquad (i = 1, 2).
\]
Then
\[
 \||\nabla|^{s}(fg)\|_{L^{r}} \lesssim \|f\|_{L^{p_{1}}} \||\nabla|^{s}g\|_{L^{q_{1}}} + \||\nabla|^{s}f\|_{L^{p_{2}}}\|g\|_{L^{q_{2}}}
\]
holds.
\end{app}

\begin{app}[Fractional chain rule for a H\"older continuous function,\ \cite{visan}]\label{chain}
Let $G$ be a H\"older continuous function of exponent $0 < \alpha < 1$.
Then, for any $0 < s < \alpha$, $1 < p < \infty$, and $\frac{s}{\alpha} < \sigma < 1$,
\[
 \||\nabla|^{s}G(u)\|_{L^{p}} \lesssim \||u|^{\alpha - \frac{s}{\sigma}}\|_{L^{p_{1}}}\||\nabla|^{\sigma}u\|_{L^{\frac{s}{\sigma}p_{2}}}^{\frac{s}{\sigma}}
\]
holds, where
\[
\frac{1}{p} = \frac{1}{p_{1}} + \frac{1}{p_{2}},\qquad \left(1 - \frac{s}{\alpha \sigma}\right)p_{1} > 1.
\]
\end{app}

\section*{Acknowledgements}
This paper is based on the author's master's thesis completed at Osaka University. I would like to express my sincere gratitude to Associate Professor Haruya Mizutani for his careful guidance and support.

\end{document}